\documentclass[preprint,number]{elsarticle}
\usepackage[utf8]{inputenc}
\usepackage{amsmath}
\usepackage{amsmath}
\usepackage{amsthm}
\usepackage{amssymb}
\usepackage{bbm}
\usepackage{verbatim}
\usepackage{graphicx}
\usepackage{subfigure}
\usepackage{afterpage}
\usepackage{etoolbox}
\usepackage{float}
\usepackage{algorithmic}
\usepackage[algosection,algoruled,vlined]{algorithm2e}
\usepackage{url}

\numberwithin{equation}{section}
\numberwithin{figure}{section}
\numberwithin{table}{section}

\makeatletter
\renewcommand{\p@subfigure}{\thefigure}
\makeatother

\newtheorem{definition}{Definition}[section]
\newtheorem{theorem}{Theorem}[section]

\newtheorem{lemma}[theorem]{Lemma}
\newtheorem{corollary}[theorem]{Corollary}
\newtheorem{rem}[theorem]{Remark}

\newcommand{\repeatable}[2]{\makeatletter \global\expandafter\def\csname repText@#1\endcsname {#2} \makeatother #2}
\newcommand{\repeatxt}[1]{\makeatletter \expandafter\csname repText@#1\endcsname \makeatother}

\newtoggle{citeInAbstract}
\toggletrue{citeInAbstract}

\newcommand{\usecrop}[2]
{
	\newlength{\cropwidth}
	\setlength{\cropwidth}{\the\textwidth}
	\addtolength{\cropwidth}{#1}
	\newlength{\cropheight}
	\setlength{\cropheight}{\the\textheight}
	\addtolength{\cropheight}{#2}
	\usepackage[width=\the\cropwidth,height=\the\cropheight,center]{crop}
}

\DeclareMathAlphabet{\mathpzc}{OT1}{pzc}{m}{it}

\usecrop{1.69in}{1.5in}

\begin{document}
	
\begin{frontmatter}
	
\title{Randomized LU Decomposition}



\author[EE]{Gil Shabat}
\author[CS]{Yaniv Shmueli}
\author[MT]{Yariv Aizenbud}
\author[CS]{Amir Averbuch\corref{maincor}}\ead{amir@math.tau.ac.il}

\address[EE]{School of Electrical Engineering, Tel Aviv University, Tel Aviv 69978, Israel}
\address[CS]{School of Computer Sciene, Tel Aviv University, Tel Aviv 69978, Israel}
\address[MT]{Department of Applied Mathematics, School of Mathematical Sciences, Tel Aviv University, Tel Aviv 69978, Israel}

\cortext[maincor]{Amir Averbuch, Tel: +972-54-5694455, Fax: +972-3-6422020}


\begin{abstract}
Randomized algorithms play a central role in low rank approximations of large matrices. In this paper, 
the scheme of the randomized SVD is extended to a randomized LU algorithm.
Several error bounds are introduced, that are based on recent results from random matrix theory related
to subgassian matrices. The bounds also improve the existing bounds of already known randomized algorithm for low rank approximation.
The algorithm is fully parallelized and thus can utilize efficiently GPUs without any CPU-GPU data transfer.
Numerical examples, which illustrate the performance of the algorithm 
and compare it to other decomposition methods, are presented.
\end{abstract}


\end{frontmatter}

\section{Introduction}
\label{sec:introduction} 
Matrix factorizations and their relations to low rank approximations play a major role in many of today's applications \cite{stewart2000decompositional}. 
In mathematics, matrix decompositions are used for low rank matrix approximations 
that often reveal interesting properties in a matrix. 
Matrix decompositions are used for example in solving linear equations and in finding least squares solutions.
In engineering, matrix decompositions are used in computer vision \cite{elad2006image}, machine learning \cite{mazumder2010spectral}, 
collaborative filtering and Big Data analytics \cite{koren2009matrix}. 
As the size of the data grows exponentially, feasible methods for the analysis of large datasets has gained an increasing interest. 
Such an analysis can involve a factorization step of the input data given as a large sample-by-feature matrix 
or given by a sample affinity matrix \cite{wolf2003learning,COIFMAN2006a,shabataccelerating}.
High memory consumption and the computational complexity of the factorization step are two main reasons for the difficulties in analyzing huge data structures.
Recently, there is an on-going interest in applying mathematical tools that are based on randomization  to overcome these difficulties. 

Some of the randomized algorithms use random projections that project the matrix to a set of random vectors. 
Formally, given a matrix $A$ of size $m \times n$ (assume $m \ge n$) and a random matrix $G$ of size $n \times k$, 
then the product $AG$ is computed to obtain a smaller matrix that potentially captures most of the range of $A$.
In most of these applications, $k$ is set to be much smaller than $n$  to obtain a compact approximation for $A$.

In this paper, we develop a randomized version of the LU decomposition.
Given an $m \times n$ matrix $A$, we seek a lower triangular $m \times k$ matrix $L$ 
and an upper triangular $k \times n$ matrix $U$ such that
\begin{equation}
\label{eq:lu_error}
\Vert LU-PAQ \Vert_2 = C(m,n,k) \sigma_{k+1}(A),
\end{equation}
where $P$ and $Q$ are orthogonal permutation matrices, 
$\sigma_{k+1}(A)$ is the $k+1$ largest singular value of $A$ and $C(m,n,k)$ is a constant that depends  on $m,n$ and $k$.

The interest in a randomized LU decomposition can be motivated (computationally wise)  by two important properties of the classical LU decomposition:
First, it can be applied efficiently to sparse matrices with computation time that depends on the number of non-zero elements. LU decomposition with full pivoting on sparse matrices can generate large regions of 
zeros in the factorized matrices \cite{schenk2000efficient, demmel1999supernodal, davis1997unsymmetric}. Processing of sparse matrices will be treated in a separate paper.
Second, LU decomposition can be fully parallelized \cite{golub2012matrix} which makes it applicable for  running on Graphics Processing Units (GPU).
GPUs are mostly used for computer games, graphics and visualization 
such as movies and 3D display. 
Their powerful computation capabilities can be used for fast matrix computations \cite{kirk2007nvidia}. 

The contributions of the paper are twofold:
A randomized version for LU decomposition, which is based on the randomized SVD template \cite{randecomp,halko2011finding}, is presented. The algorithm is analyzed
and several error bounds are derived. The bounds are based on recent results from random matrix theory for the largest and smallest singular values of random matrices
with subgaussian entries \cite{litvak2005smallest,litvak2010}. This technique is also used to improve the bounds for the randomized SVD.
The randomized LU is fully implemented to run on a standard GPU card without any GPU-CPU data transfer. It enables us to accelerate the algorithm significantly.
We present numerical results that compare our algorithm with other decomposition methods and show that it outperforms them.

The paper is organized as follows: Section \ref{sec:related_work}, overviews 
related work on matrix decomposition and approximation that use randomized methods.
Section \ref{sec:perliminaries} reviews several mathematical results that are needed
for the development of the randomized LU.
Section \ref{sec:random_lu} presents several randomized LU algorithms
and several error bounds  on their approximations are proved.
Section \ref{sec:numerical_results} presents numerical results on the approximation error,
the computational complexity of the algorithm and compares it with other methods.
The performance comparison was done on different types of matrices and by using GPU cards.

\section{Related Work}
\label{sec:related_work}
Efficient matrix decomposition serves as a basis for many
studies and algorithms design for data analysis and applications. Fast randomized matrix decomposition algorithms are used for tracking objects in videos \cite{shabataccelerating}, 
multiscale extensions for data \cite{bermanis2011multiscale} 
and detecting anomalies in network traffic for finding cyber attacks \cite{gild:PHD}, to name some. 
There are randomized versions for many different matrix factorization algorithms \cite{halko2011finding}, compressed sensing \cite{donoho2006compressed} and  least squares \cite{avron2010blendenpik}.

There is a variety of methods and algorithms that factorize a matrix into several matrices.
Typically, the factorized terms have properties such as being triangular, orthogonal, diagonal,
sparse or low rank.
In general, a certain control on the desired approximation error for a factorized matrix is possible. For example, it is achievable by increasing the rank of a low rank approximation or by allowing dense factors for sparse decompositions.

Rank revealing factorization uses permutation matrices on the columns and rows of a martrix $A$
so that the factorized matrices structure have a strong rank portion and a rank deficient portion.
The most known example for approximating an $m\times n$ matrix $A$ by a low rank $k$ matrix is the truncated SVD.
Other rank revealing factorizations can be used to achieve low rank approximations.
For example, both QR and LU factorizations have rank revealing versions  such as
RRQR decomposition~\cite{chan1987rank}, strong RRQR~\cite{gu1996efficient} decomposition,
RRLU decomposition~\cite{pan2000existence}
and strong RRLU decomposition~\cite{miranian2003strong}.

Other matrix factorization methods such as Interpolative Decomposition (ID) \cite{low_rank_comp}  and
CUR decomposition \cite{drineas2008relative}, use columns and rows of the original  matrix $A$ in the factorization process.
Such a property  exposes the most important portions that construct $A$.
An ID factorization of order $k$ of an $m\times n$ matrix $A$  consists of an $m \times k$ matrix
$B$ whose columns consist of a subset of the columns of $A$, as well
as a $k \times n$ matrix $P$, such that a subset of the columns
of $P$ becomes a $k\times k$ identity matrix and $A\approx BP$ such that $\Vert A-BP\Vert\lesssim\mathcal{O}(n,\sigma_{k+1}(A))$.
Usually, $k=\#\{j:\sigma_j(A)\geq\delta\sigma_1(A)\}$ is the numerical rank
of $A$ up to a certain accuracy $\delta>0$.
This selection of $k$
guarantees that the columns of $B$ constitute a well-conditioned
basis for the range of $A$ \cite{low_rank_comp}.

Randomized version for many important algorithms have been developed in order to reduce the computational complexity by approximating
the solution to a desired rank.
These include SVD, QR and ID factorizations \cite{randecomp},
CUR decomposition as a randomized version \cite{drineas2008relative} of the
pseudo-skeleton decomposition, methods for solving least squares
problems \cite{rokhlin2008fast, clarkson2013low, avron2010blendenpik} and low rank approximations \cite{clarkson2013low,achlioptas2007fast}.

In general, randomization methods for matrix factorization  have two steps:
1. A low-dimensional space, which captures most of the ``energy" of $A$, is found using randomization.
2. $A$ is projected onto the retrieved subspace and the projected matrix is factorized~\cite{halko2011finding}.

Several different options exist when random projection matrix is used in the step 1. For example, it can
be a matrix of random signs ($\pm 1$) \cite{clarkson2009numerical,magen2011low}, a matrix of i.i.d Gaussian random variables 
with zero mean and unit variance \cite{randecomp},
a matrix whose columns  are selected randomly from the identity matrix with either uniform or 
non-uniform probability \cite{frieze2004fast,drineas2006fast},
a random sparse matrix designed to enable fast multiplication with a sparse input matrix \cite{clarkson2013low,achlioptas2007fast},
random structured matrices, which use orthogonal transforms such as discrete Fourier transform, 
Walsh-Hadamard transform and more (\cite{rokhlin2008fast,avron2010blendenpik,boutsidis2013improved}). 
In our algorithm, we use Gaussian matrices in Step 1 as well as 
structured Fourier matrices to achieve accelerated computation.

\section{Preliminaries}
\label{sec:perliminaries}
In this section, we review the rank revealing LU (RRLU) decomposition and bounds on singular values bounds for random matrices 
that will be used to prove the error bounds for the randomized LU algorithm. 
Throughout the paper, we use the following notation: for any matrix $A$, $\sigma_j(A)$ is 
the $j$th largest singular value and $\Vert A \Vert$ is the spectral norm (the largest singular value or $l_2$ operator norm). 
If $x$ is a vector then $\Vert x \Vert$ is the standard $l_2$ (Euclidean) norm. 
$A^{\dagger}$ denotes the pseudo-inverse of $A$. For a random variable $X$, $\mathbb{E}$ denotes the  expectation
of $X$ and $\mathbb{P}(X \ge x)$ is the probably of a random variable $X$ to be larger than a scalar $x$.
\subsection{Rank Revealing LU (RRLU)}
The following theorem is adapted from \cite{pan2000existence} (Theorem 1.2):
\begin{theorem}[\cite{pan2000existence}]
\label{trm:rrlu-pan}
Let $A$ be an $m\times n$ matrix ($m \ge n$). Given an integer $1 \le k < n$, the following factorization 
\begin{equation}
\label{eq:rrlu_def}
PAQ = \begin{pmatrix} L_{11} & 0 \\ L_{21} & I_{n-k} \end{pmatrix}
\begin{pmatrix} U_{11} & U_{12} \\ 0 & U_{22} \end{pmatrix}
\end{equation}
holds where $L_{11}$ is a unit lower triangular, $U_{11}$ is an upper triangular, $P$ and $Q$ are orthogonal permutation matrices. 
Let $\sigma_1 \ge \sigma_2 \ge ... \ge \sigma_n \ge 0$ be the singular values of $A$, then
\begin{equation}
\label{eq:rrlu_def1}
\sigma_k \ge \sigma_{min}(L_{11}U_{11}) \ge \frac{\sigma_k}{k(n-k)+1},
\end{equation}
and 
\begin{equation}
\label{eq:rrlu_def2}
\sigma_{k+1} \le \Vert U_{22} \Vert \le (k(n-k)+1)\sigma_{k+1}.
\end{equation}
\end{theorem}

This is called RRLU decomposition. Based on Theorem \ref{trm:rrlu-pan}, we have the following definition: 
\begin{definition}[RRLU Rank $k$ Approximation denoted RRLU$_k$] 
Given a RRLU decomposition (Theorem \ref{trm:rrlu-pan}) of a matrix $A$ with an integer $k$ (as in Eq. \eqref{eq:rrlu_def})
such that $PAQ=LU$. The RRLU rank $k$ approximation is defined by taking $k$ columns 
from $L$ and $k$ rows from $U$ such that
\begin{equation}
\label{eq:rrlu_rank_k_approx}
\text{RRLU}_k(PAQ)=\begin{pmatrix} L_{11} \\ L_{21} \end{pmatrix}
\begin{pmatrix} U_{11} U_{12} \end{pmatrix}
\end{equation}
where $L_{11}, L_{21}, U_{11}, U_{12}, P$ and $Q$ are defined in Theorem \ref{trm:rrlu-pan}.
\end{definition}

\begin{lemma}[RRLU Approximation Error] 
The error of the RRLU$_k$ approximation of $A$ is
\label{lem:rrlu_approx_err}
\begin{equation}
\Vert PAQ-\text{RRLU}_k(PAQ) \Vert \le (k(n-k)+1)\sigma_{k+1}.
\end{equation}
\end{lemma}

\begin{proof}
The proof follows directly from Eqs. \eqref{eq:rrlu_def} and \eqref{eq:rrlu_rank_k_approx}.
\end{proof}

Lemma \ref{lem:bhatia} appears in \cite{bhatia1997matrix}, page 75:
\begin{lemma}[\cite{bhatia1997matrix}]
\label{lem:bhatia}
Let $A$ and $B$ be two matrices and let $\sigma_j(\cdot)$ denotes the $j$th singular value of a matrix. Then,
$ \sigma_j(AB) \le \Vert A \Vert \sigma_j(B) \nonumber$ and $\sigma_j(AB) \le \Vert B \Vert \sigma_j(A) \nonumber$.
\end{lemma}

Lemma \ref{lem:bigval} was taken from \cite{randecomp} and it is an 
equivalent formulation for Eq. 8.8 in \cite{goldstine1951numerical}.
\begin{lemma}[\cite{randecomp}]
\label{lem:bigval}
Suppose that $G$ is a real $n \times l$ matrix whose entries are i.i.d 
Gaussian random variables with zero mean and unit variance and let $m$ be an integer such that $m\ge l$, $m\ge n$, $\gamma>1$ and
\begin{equation}
\label{eq:bigval}
1-\frac{1}{4(\gamma^2-1)\sqrt{\pi m\gamma^2}}\left( \frac{2\gamma^2}{e^{\gamma^2-1}}\right)^m \ge 0.
\end{equation}
Then, $\Vert G \Vert \le \sqrt{2m}\gamma$ with probability not less than the value in Eq. \eqref{eq:bigval}.
\end{lemma}

\subsection{Subgaussian Random Matrices}
\label{sec:sparse}

\begin{definition}
\label{def:bsubgauss}
A real valued random variable $X$ is called subgaussian if there exists $b>0$ such that for all $t>0$  we have
$\mathbb{E}e^{tX} \le e^{b^2t^2/2}$.
\end{definition}


We review several results adapted from \cite{litvak2005smallest,rudelson2009smallest} about random matrices whose entries are subgaussian.
We focus on the case where $A$ is a tall $m \times n$ matrix ($m>(1+\frac{1}{\ln n})n$). 
Similar results can be found in \cite{litvak2010} for square and almost square matrices. 

\begin{definition}
\label{def:mat_family_A}
Assume that $\mu \ge 1$, $a_1>0$ and $a_2>0$. $\mathcal{A}(\mu,a_1,a_2,m,n)$ is the set of all $m \times n$ ($m>n$) 
random matrices $A=(\xi_{ij})$ whose entries are centered i.i.d real valued random variables satisfying the following conditions:

\begin{enumerate}
\item
Moments: $\mathbb{E} \vert \xi_{ij} \vert ^3 \le \mu^3$;
\item
Norm: $\mathbb{P}(\Vert A \Vert > a_1\sqrt{m}) \le e^{-a_2 m}$ where $\mathbb{P}$ is a probability function;
\item
Variance: $\mathbb{E} \xi_{ij}^2 \ge 1$.
\end{enumerate}
\end{definition}
It is shown in \cite{litvak2005smallest} that if $A$ is subgaussian then $A \in \mathcal{A}$. 
For a Gaussian matrix with zero mean and unit variance,~$\mu=\left(\frac{4}{\sqrt{2\pi}}\right)^\frac{1}{3}$

Theorems \ref{trm:subgauss} and \ref{trm:smallest_sv} are taken from Section 2 in \cite{litvak2005smallest}.
\begin{theorem}[\cite{litvak2005smallest}]
\label{trm:subgauss}
Every matrix $A$ of size $m\times n$ ($m \ge n$), whose entries are subgaussian with $\mu \ge 1$ and $a_2\ge 0$, satisfies:
\begin{equation}
\mathbb{P}\left( \Vert A \Vert \ge a_1\sqrt{m}\right) \le e^{-a_2 m}
\end{equation}
where $a_1=6\mu\sqrt{a_2+4}$.
\end{theorem}

Theorem \ref{trm:subgauss} provides an upper bound for the largest singular value that depends on 
the desired probability. 
Theorem \ref{trm:smallest_sv} is used to 
bound from below the smallest singular value  of random Gaussian matrices.
\begin{theorem}[\cite{litvak2005smallest}]
\label{trm:smallest_sv}
Let $\mu \ge 1$, $a_1,a_2>0$. Let $A$ be an $m \times n$ matrix where $m>(1+\frac{1}{\ln n})n$. $m$ can be written as $m=(1+\delta)n$. 
Suppose that the entries of $A$ are centered independent random variables such that conditions $1,2,3$ in Definition \ref{def:mat_family_A} hold. 
Then, there exist positive constants $c_1$ and $ c_2$ such that
\begin{equation}
\label{smallsveq}
\mathbb{P}(\sigma_n(A) \le c_1\sqrt{m}) \le e^{-m}+e^{-c''m/(2\mu^6)}+e^{-a_2 m} \le e^{-c_2 m}.
\end{equation}
\end{theorem}
From Theorem \ref{trm:smallest_sv}, the exact values of constants $c_1, c_2$ and $c''$ are
\begin{equation}
\label{c1_const}
c_1=\frac{b}{e^2 c_3}\left( \frac{b}{3e^2 c_3 a_1} \right)^\frac{1}{\delta},~~c''=\frac{27}{2^{11}}
\end{equation}
where $c_3=4\sqrt{\frac{2}{\pi}}\left(\frac{2\mu^9}{a_1^3}+\sqrt{\pi}\right)$, 
$b=\min \left( \frac{1}{4}, \frac{c'}{5a_1\mu^3}\right)$ 
and $c'=\left( \frac{27}{2^{13}}\right)^\frac{1}{2}$.
For the constant $c_2$, we need a small enough constant to satisfy 
the inequality in Eq. \eqref{smallsveq} and set it, for simplification, to
\begin{equation}
\label{c2_const}
c_2=\min\left(1, \frac{c''}{(2\mu^6)}, a_2\right)-\frac{\ln3}{m}.
\end{equation}
The setting of $c_2$ according to Eq. \eqref{c2_const} comes from a relaxation of the inequality
\begin{equation*}
	e^{-m}+e^{-c''m/(2\mu^6)}+e^{-a_2 m} \le 3e^{-\min\left(1, \frac{c''}{(2\mu^6)}, a_2\right)} \le e^{-c_2 m}
\end{equation*}
and solving $3e^{-\min\left(1, \frac{c''}{(2\mu^6)}, a_2\right)} \le e^{-c_2 m}$ for $c_2$.

\subsection{The SRFT matrix} \label{subsec:srht}
The Subsampled Random Fourier Transform (SRFT), which is described in
\cite{AC,WLRT}, is a random matrix $R$ with the structure $R = DFS$
where $D$ is an $n \times n$ diagonal matrix whose entries are
i.i.d. random variables drawn from a uniform distribution on the
unit circle in $\mathbb{C}$, $F$ is an $n\times n$ discrete Fourier
transform such that $F_{jk}= \frac{1}{\sqrt{n}}e^{-2\pi i
	(j-1)(k-1)/n}$ and $S$ is an $n\times l$ matrix whose entries are all
zeros except for a single randomly placed 1 in each column.

\begin{lemma}[\cite{WLRT}] \label{lem:SRFTmult}
	For any $m \times n$ matrix $A$, let $R$ be the $n \times l$  SRFT matrix. Then, $Y=AR$ can be computed in $\mathcal{O}(mn\log l)$  floating point operations.
\end{lemma}

\subsection{Interpolative decomposition (ID)}
\label{subsec:id}
Let $ A$ be an $ m \times n $ of rank $r$.  $ A \approx A_{(:,J)} X$ is
the ID of rank $r$ of $A$ if:
\begin{enumerate}
	\item $ J $ is a subset of $ r $ indices from $1,\ldots, n$.
	\item The $ r \times n $ matrix $ A_{(:,J)} $ is a subset of $J$ columns from $A$.
	\item $ X $ is an $ r \times n $ matrix whose entries are less than 2 in magnitude and contains $r$ columns of the identity matrix.
\end{enumerate}
Similarly, it is possible to compute the ID with row selection such that $A \approx XA_{(J,:)}$. The ID is based on \cite{gu1996efficient} and it is introduced 
in \cite{randecomp,low_rank_comp,halko2011finding} for deterministic and random
algorithms. It is possible to compute ID with LU instead of using QR. This can increase the reconstruction error, since RRQR has better bounds than RRLU (\cite{gu1996efficient,pan2000existence}) while reducing the computational complexity since LU is faster to compute than QR (\cite{golub2012matrix}). 

\section{Randomized LU}
\label{sec:random_lu}
In this section, we present the randomized LU algorithm (Algorithm \ref{alg:randomized_lu}) that computes the LU rank $k$ 
approximation of a full matrix. 
In addition, we present Algorithm \ref{alg:fast_randomized_lu} that utilizes the SRFT matrix for achieving a faster processing.
Error bounds are derived for each algorithm. 

The algorithm begins by projecting the input matrix on a random matrix. 
The resulting matrix captures most of the range of the input matrix.
Then, we compute a triangular basis for this matrix and project the input matrix on it. 
Finally, we find a second triangular basis for the projected columns and multiply it with the original basis.
The product leads to a lower triangular matrix 
$L$ and the upper triangular matrix $U$ is obtained from the second LU factorization.

\begin{algorithm}[H]
\caption{Randomized LU Decomposition}
\label{alg:randomized_lu}
\textbf{Input:} $A$ matrix of size $m \times n$ to decompose, $k$  desired rank, $l \ge k$ number of columns to use.\\
\textbf{Output:} Matrices $P,Q,L,U$ such that $\Vert PAQ-LU \Vert \le \mathcal{O}(\sigma_{k+1}(A))$ where $P$ and $Q$ 
are orthogonal permutation matrices, $L$ and $U$ are the lower and upper triangular matrices, respectively.
\begin{algorithmic}[1]
\STATE Create a matrix $G$ of size $n \times l$ whose entries are i.i.d. Gaussian random variables 
with zero mean and unit standard deviation.
\STATE $Y \gets AG$.
\STATE Apply RRLU decomposition (Theorem \ref{trm:rrlu-pan}) to $Y$ such that $PYQ_y=L_yU_y$.
\STATE Truncate $L_y$ and $U_y$ by choosing the first $k$ columns and the first $k$ rows, respectively, such that
$L_y \leftarrow L_y(:,1:k)$ and $U_y \leftarrow U_y(1:k,:)$.
\STATE $B \gets L_y^{\dagger}PA$.
\STATE Apply LU decomposition to $B$ with column pivoting $BQ=L_bU_b$.
\STATE $L \gets L_y L_b$.
\STATE $U \gets U_b$.
\end{algorithmic}
\end{algorithm}

\begin{rem}
The pseudo-inverse of $L_y$ in step $5$ can be 
computed by $L_y^{\dagger}=(L_y^T L_y)^{-1}L_y^T$. 
This can be done efficiently when it is computed on platforms such as GPUs that can multiply matrices  via parallelization.
Usually, the inversion is done on a small matrix since in many cases $k \ll n$ and therefore it can be done cheaply (computationally wise)
by the application of Gaussian elimination.
\end{rem}

\begin{rem} In practice, it is sufficient to perform step $3$ in Algorithm \ref{alg:randomized_lu} 
using standard LU decomposition with partial pivoting instead of applying RRLU. The cases where $U$ grows exponentially
are extremely rare -- see section 3.4.5 in \cite{golub2012matrix,trefethen1990average}.
\end{rem}

Theorem \ref{trm:rand_lu_err} presents an error bound for Algorithm \ref{alg:randomized_lu}:
\begin{theorem} 
\label{trm:rand_lu_err}
Let $A$ be a matrix of size $m\times n$. Then, its randomized LU decomposition 
produced by Algorithm \ref{alg:randomized_lu} with integers $k$ and $l$  ($l\ge k$)  satisfies:

\begin{equation}
\Vert LU-PAQ \Vert \le \left(2\sqrt{2nl\beta^2\gamma^2+1}+2\sqrt{2nl}\beta\gamma \left( k(n-k)+1 \right)\right)\sigma_{k+1}(A) ,
\end{equation}
with probability not less than
\begin{equation}
\label{eq:rand_lu_err_prob}
\xi \triangleq 1-\frac{1}{\sqrt{2\pi(l-k+1)}}\left(\frac{e}{(l-k+1)\beta}\right)^{l-k+1}-
\frac{1}{4(\gamma^2-1)\sqrt{\pi n \gamma^2}}\left(\frac{2\gamma^2}{e^{\gamma^2-1}}\right)^n ,
\end{equation}
where $\beta>0$ and $\gamma>1$.
\end{theorem}

The proof of Theorem \ref{trm:rand_lu_err} is given in  Section \ref{sec:bounds_rand_lu}.
To show that the success probability $\xi$ in Eq. \eqref{eq:rand_lu_err_prob} is sufficiently high, 
we present several calculated values of $\xi$ in Table \ref{table:ran_lu_err_values}. 
We omitted the value of $n$ from Table \ref{table:ran_lu_err_values} since it does not affect the value of $\xi$
due to the fact that the second term in Eq. \eqref{eq:rand_lu_err_prob} decays fast.

\begin{table}[H]
\centering
\caption{Calculated values for the success probability $\xi$ (Eq. \eqref{eq:rand_lu_err_prob}). 
The terms $l-k$, $\beta$  and $\gamma$ appear in Eq.~\ref{eq:rand_lu_err_prob}.}
\label{table:ran_lu_err_values}
\begin{tabular}{c c c c}
\hline 
$l-k$  & $\beta$  & $\gamma$  & $\xi$  
\rule{0pt}{2.6ex} \rule[-1.2ex]{0pt}{0pt}\\
[0.5ex] \hline
3 &	 5 &	 5 &	 $1 - 6.8\times 10^{-5}$ \\
5 &	 5 &	 5 &	 $1 - 9.0\times 10^{-8}$ \\
10 &	 5 &	 5 &	 $1 - 5.2 \times 10^{-16}$ \\
3 &	 30 &	 5 &	 $1 - 5.2 \times 10^{-8}$ \\
5 &	 30 &	 5 &	 $1 - 1.9 \times 10^{-12}$ \\
10 &	 30 &	 5 &	 $1 - 1.4 \times 10^{-24}$ \\
3 &	 30 &	 10 &	 $1 - 5.2\times 10^{-8}$ \\
5 &	 30 &	 10 &	 $1 - 1.9\times 10^{-12}$ \\
10 &	 30 &	 10 &	 $1 - 1.4\times 10^{-24}$ \\
[1ex] \hline
\end{tabular}
\end{table}

In Section \ref{sec:numerical_results}, we show that in practice, Algorithm \ref{alg:randomized_lu} produces comparable 
results to other well-known randomized factorization methods of low rank matrices such as randomized SVD and randomized ID.

\subsection{Computational Complexity Analysis}
\label{sec:comp_cost}
To compute the number of floating points operations in Algorithm \ref{alg:randomized_lu}, we evaluate the complexity of each step:
\begin{description}
\item[Step 1:] Generating an $n \times l$ random matrix requires $\mathcal{O}(nl)$ operations.
\item[Step 2:] Multiplying $A$ by $G$ to form $Y$ requires $l\mathcal{C}_A$ operations, where $\mathcal{C}_A$ is the 
complexity of applying $A$ to an $n \times 1$ column vector. 
\item[Step 3:] Partial pivoting computation of LU for $Y$ requires $\mathcal{O}(ml^2)$ operations.
\item[Step 4:] Selecting the first $k$ columns (we do not modify them) requires $\mathcal{O}(1)$ operations.
\item[Step 5:] Computing the pseudo inverse of $L_y$ requires $\mathcal{O}(k^2m+k^3+k^2m)$ operations and 
multiplying it by $A$ requires $k\mathcal{C}_{A^T}$ operations. 
Note that $P$ is a permutation matrix that does not modify the rows of $A$.
\item[Step 6:] Computing the partial pivoting LU for $B$ requires $\mathcal{O}(k^2n)$ operations.
\item[Step 7:] Computing $L$ requires $\mathcal{O}(k^2 m)$ operations.
\item[Step 8:] Computing $U$ requires $\mathcal{O}(1)$ operations.
\end{description}
By summing up the complexities of all the steps above, then Algorithm \ref{alg:randomized_lu} necessitated 
\begin{equation}
\mathcal{C}_{Rand LU}=l\mathcal{C}_A+k\mathcal{C}_{A^T}+\mathcal{O}(l^2m+k^3+k^2n)
\end{equation}
operations.
Here, we used $C_A$ (and $C_{A^T}$) to denote the complexity from the application of $A$ (and $A^T$) to  a vector, respectively.
For a general $A$, $\mathcal{C}_A=\mathcal{C}_{A^T}=\mathcal{O}(mn)$.

\subsection{Bounds for the Randomized LU (Proof of Theorem~\ref{trm:rand_lu_err}) }
\label{sec:bounds_rand_lu}
In this section, we prove Theorem \ref{trm:rand_lu_err} and provide an additional complementary bound. 
This is done by finding a basis to the smaller matrix $AG$, which is achieved in practice by using RRLU. 
The assumptions are that $L$ is numerically stable so its pseudo-inverse can be computed accurately, 
 there exists a matrix $U$ such that $LU$ is a good approximation to $AG$ and  there exists a 
matrix $F$ such that $\Vert AGF-A\Vert$ is small. $L$ 
 is always numerically stable since it has a small condition number \cite{stewart1998triangular}.

Lemmas \ref{lem:rand_lu_bound1},\ref{lem:rand_gauss} and \ref{lem:rand_gauss2} are needed for the proof of Theorem \ref{trm:rand_lu_err}. 
Lemma \ref{lem:rand_lu_bound1} states that a given basis $L$ can form a basis for the columns in $A$ by
bounding the error $\Vert LL^\dagger A-A \Vert$. 

\begin{lemma} 
\label{lem:rand_lu_bound1}
Assume that $A$ is an $m \times n$ matrix, $L$ is an $m \times k$ matrix with rank $k$, $G$ is an $n \times l$ matrix, $l$ is an integer ($l \ge k$), $U$ is a $k \times l$ matrix
and $F$ is $l \times n$ ($k\le m$) matrix. Then,
\begin{equation}
\Vert LL^\dagger A-A \Vert \le 2\Vert AGF-A\Vert+2\Vert F \Vert \Vert LU-AG \Vert.
\end{equation}
\end{lemma}
\begin{proof}
By using the triangular inequality we get
\begin{equation}
\label{eq:rand_lu_bound1}
\Vert LL^\dagger A-A \Vert \le \Vert LL^\dagger A-LL^\dagger AGF \Vert+\Vert LL^\dagger AGF-AGF \Vert+\Vert AGF-A \Vert.
\end{equation}
Clearly, the first term can also be bounded by
\begin{equation}
\label{eq:rand_lu_bound11}
\Vert LL^\dagger A-LL^\dagger AGF \Vert \le \Vert LL^\dagger \Vert \Vert A-AGF \Vert \le \Vert A-AGF \Vert .
\end{equation}
The second term can be bounded by
\begin{equation}
\label{eq:rand_lu_bound2}
\Vert LL^\dagger AGF-AGF \Vert \le \Vert F\Vert \Vert LL^\dagger AG-AG \Vert .
\end{equation}
In addition,
\begin{equation}
\label{eq:rand_lu_bound3}
\Vert LL^\dagger AG-AG \Vert \le  \Vert LL^\dagger AG-LL^\dagger LU \Vert+\Vert LL^\dagger LU-LU\Vert+\Vert LU-AG\Vert.
\end{equation}
Since $L^\dagger L=I$, it follows that $\Vert LL^\dagger LU-LU \Vert=0$ and that 
$\Vert LL^\dagger AG-LL^\dagger LU \Vert \le \Vert AG-LU \Vert$. When combined with Eq. \eqref{eq:rand_lu_bound3} we obtain:
\begin{equation}
\label{eq:rand_lu_bound4}
\Vert LL^\dagger AG-AG \Vert \le 2 \Vert LU-AG\Vert.
\end{equation}
By substituting Eq. \eqref{eq:rand_lu_bound4} in Eq. \eqref{eq:rand_lu_bound2} we get
\begin{equation}
\label{eq:rand_lu_bound_final}
\Vert LL^\dagger AGF-AGF \Vert \le 2\Vert F \Vert \Vert LU-AG\Vert.
\end{equation}
By substituting Eqs. \eqref{eq:rand_lu_bound11} and \eqref{eq:rand_lu_bound_final} in Eq. \eqref{eq:rand_lu_bound1} we get
\begin{equation}
\Vert LL^\dagger A-A \Vert \le 2\Vert AGF-A\Vert+2\Vert F \Vert \Vert LU-AG \Vert.
\end{equation}
\end{proof}

Lemma \ref{lem:rand_gauss} appears in \cite{randecomp}. It uses a lower bound for the smallest singular value of a 
Gaussian matrix with zero mean and unit variance. This bound appears in \cite{chen2005condition}.
\begin{lemma}[\cite{randecomp}]
\label{lem:rand_gauss}
Assume that $k,l,m$ and $n$ are positive integers such that $k\le l$, $l \le \min{(m,n)}$. 
Assume that $A$ is a real $m \times n$ matrix, $G$ is $n \times l$ matrix whose entries are i.i.d Gaussian random variables 
of zero mean and unit variance, $\beta$ and $\gamma$ are real numbers, such that $\beta>0$, $\gamma>1$ and the quantity
\begin{equation}
\label{eq:rand_gauss_quantity}
1-\frac{1}{\sqrt{2\pi(l-k+1)}}\left(\frac{e}{(l-k+1)\beta}\right)^{l-k+1}-\frac{1}{4(\gamma^2-1)\sqrt{\pi n \gamma^2}}\left(\frac{2\gamma^2}{e^{\gamma^2-1}}\right)^n
\end{equation}
is non-negative. Then, there exists a real $l \times n$ matrix $F$ such that
\begin{equation}
\Vert AGF-A \Vert \le \sqrt{2nl\beta^2\gamma^2+1} \sigma_{k+1}(A) 
\end{equation}
and 
\begin{equation}
\Vert F \Vert \le \sqrt{l}\beta
\end{equation}
with probability not less than the value in Eq. \eqref{eq:rand_gauss_quantity}.
\end{lemma}

Lemma \ref{lem:rand_gauss2} rephrases Lemma \ref{lem:rand_gauss} 
by utilizing the bounds that appear in Section \ref{sec:sparse}.
The proof is close to the argumentation that appears in the 
proof of Lemma \ref{lem:rand_gauss}.

\begin{lemma}
\label{lem:rand_gauss2}
Let $A$ be a real $m\times n$ ($m \ge n$) matrix.
Let $G$ be a real $n \times l$ matrix whose entries are Gaussian i.i.d with zero mean and unit variance. 
Let $k$ and $l$ be integers such that $l<\min{(m,n)}$ and $l > \left( 1+\frac{1}{\ln k} \right)k$. 
We define $a_1,a_2,c_1$ and $c_2$ as in Theorem \ref{trm:smallest_sv}. Then, there exists a real matrix $F$ of size $l \times n$ such that
\begin{equation}
\Vert AGF-A \Vert \le \sqrt{\frac{a_1^2 n}{c_1^2 l}+1}\sigma_{k+1}(A),
\end{equation}
and
\begin{equation}
\Vert F \Vert \le \frac{1}{c_1\sqrt{l}}
\end{equation}
with probability not less than $1-e^{-c_2 l}-e^{-a_2 n}$.
\end{lemma}
\begin{proof}
We begin by the application of SVD to $A$ such that
\begin{equation}
\label{eq:rand_gauss21}
A=U\Sigma V^T,
\end{equation}
where $U$ is orthogonal $m\times m$ matrix, $\Sigma$ is $m\times n$ diagonal matrix with non-negative 
entries and $V$ is orthogonal $n \times n$ matrix.
Assume that given $V^T$ and  $G$, suppose that 
\begin{equation}
\label{eq:rand_gauss22}
V^T G=\begin{pmatrix} H \\ R \end{pmatrix},
\end{equation}
where $H$ is $k\times l$ matrix and $R$ is $(n-k)\times l$ matrix. 
Since $G$ is a Gaussian i.i.d. matrix and $V$ is an orthogonal matrix, then $V^T G$ is also a Gaussian i.i.d. matrix.
Therefore, $H$ is a Gaussian i.i.d. matrix.
Define $F=PV^T$, where $P$ is a matrix of size $l \times n$ such that $P=\begin{pmatrix} H^\dagger & 0.\end{pmatrix}$ Therefore, 
\begin{equation}
\label{eq:rand_gauss23}
F=\begin{pmatrix} H^\dagger & 0 \end{pmatrix}V^T.
\end{equation}
By computing $\Vert F \Vert$ using Theorem \ref{trm:smallest_sv}, we get
\begin{equation}
\label{eq:norm_bound_F}
\Vert F \Vert = \Vert PV^T \Vert=\Vert H^\dagger \Vert=\Vert H^T(HH^T)^{-1} \Vert=\frac{1}{\sigma_k(H)} \le \frac{1}{c_1\sqrt{l}}
\end{equation}
with probability not less than $1-e^{-c_2 l}$. 
Now, we can bound $\Vert AGF-A \Vert$. 
By using Eqs. \eqref{eq:rand_gauss21}, \eqref{eq:rand_gauss22} and \eqref{eq:rand_gauss23} we get
\begin{equation}
\label{eq:rand_gauss24}
AGF-A=U\Sigma\begin{pmatrix} \begin{pmatrix} H \\ R \end{pmatrix} \begin{pmatrix} H^\dagger & 0 \end{pmatrix}-I \end{pmatrix}V^T.
\end{equation}
We define $S$ to be the  $k\times k$ upper-left block of $\Sigma$. Let $T$ to be  the $(n-k)\times (n-k)$  lower-right  block. Then,
\begin{equation*}
\Sigma\begin{pmatrix} \begin{pmatrix} H \\ R \end{pmatrix} \begin{pmatrix} H^\dagger & 0 \end{pmatrix}-I \end{pmatrix}=
\begin{pmatrix} S & 0 \\ 0 & T \end{pmatrix} \begin{pmatrix} 0 & 0 \\ RH^\dagger & -I \end{pmatrix}=\begin{pmatrix} 0 & 0 \\ TRH^\dagger & -T \end{pmatrix}.
\end{equation*}
The norm of the last term is:
\begin{equation}
\label{eq:rand_gauss25}
\left\Vert \begin{pmatrix} 0 & 0 \\ TRH^\dagger & -T \end{pmatrix} \right\Vert^2 \le \Vert TRH^\dagger \Vert^2 + \Vert T \Vert^2. 
\end{equation}
Therefore, by using Eqs. \eqref{eq:rand_gauss24}, \eqref{eq:rand_gauss25} and the fact that $\Vert T \Vert = \sigma_{k+1}(A)$, we get 
\begin{equation}
\label{lem:boundAGF}
\Vert AGF-A \Vert \le \sqrt{\Vert TRH^\dagger \Vert^2 + \Vert T \Vert^2} \le \sqrt{ \Vert H^\dagger \Vert^2 \Vert R \Vert^2+1} \sigma_{k+1}(A).
\end{equation}
We also know that
\begin{equation*}
\Vert R \Vert \le \Vert V^T G \Vert =\Vert G \Vert \le a_1\sqrt{n}
\end{equation*}
with probability not less than $1-e^{-a_2 n}$. Combining Eq. \eqref{lem:boundAGF} 
with the fact that $\Vert H^\dagger \Vert \le \frac{1}{c_1\sqrt{l}}$ and $\Vert R \Vert \le a_1\sqrt{n}$ gives
\begin{equation}
\Vert AGF-A \Vert \le \sigma_{k+1}(A) \sqrt{\frac{a_1^2 n}{c_1^2 l}+1}.
\end{equation}
\end{proof}

\begin{rem}
\label{rem:asymp_bound}
In contrast to Lemma \ref{lem:rand_gauss} where $\Vert AGF-A\Vert=\mathcal{O}(\sqrt{nl})$ , 
Lemma \ref{lem:rand_gauss2} provides the bound $\Vert AGF-A\Vert=\mathcal{O}(\sqrt{\frac{n}{l}})$ that is tighter for large values of $l$.
\end{rem}
\begin{rem}
The condition $l > \left( 1+ \frac{1}{\ln k}\right)k$ in Lemma~\ref{lem:rand_gauss2} has to be satisfied to meet the error bounds. However, there are bounds for the case where $H$ is almost square ($l \approx k$) or square ($l=k$) and they are given in \cite{litvak2010}.
\end{rem}

\begin{proof}[Proof of Theorem \ref{trm:rand_lu_err}]
The error is given by  $\Vert LU-PAQ\Vert$ where $L,U,P$ and $Q$ 
are the outputs from Algorithm \ref{alg:randomized_lu} 
whose inputs are the matrix $A$ and integers $k$ and $l$.
From Steps 7 and 8 in Algorithm \ref{alg:randomized_lu}  we have
\begin{equation}
\Vert LU-PAQ \Vert=\Vert L_yL_bU_b-PAQ \Vert
\end{equation}
where $L_y$ is the $m \times k$ matrix in step 4 in Algorithm \ref{alg:randomized_lu}.
By using the fact that $BQ=L_bU_b=L_y^\dagger PAQ$, we get 
\begin{equation}
\label{eq:main_bound_proof}
\Vert LU-PAQ \Vert=\Vert L_yL_bU_b-PAQ \Vert=\Vert L_yL_y^\dagger PAQ-PAQ \Vert.
\end{equation}
The application of Lemma \ref{lem:rand_lu_bound1} to Eq. \eqref{eq:main_bound_proof} gives
\begin{equation}
\begin{array}{lll}
\label{eq:main_bound_proof2}
\Vert LU-PAQ \Vert &=& \Vert L_yL_y^\dagger PAQ-PAQ \Vert \\\\ 
& \le & 2\Vert PAQ\tilde{G}F-PAQ\Vert+2\Vert F \Vert \Vert L_yU_y-PAQ\tilde{G}\Vert
\end{array}
\end{equation}
where $U_y$ is the $k \times n$ matrix in step 4 in Algorithm \ref{alg:randomized_lu}.
This holds for any matrix $\tilde{G}$.
In particular, it holds for a  
matrix $\tilde{G}$ that satisfies $Q\tilde{G}=GQ_y$ where $G$ is a random 
Gaussian i.i.d. matrix. After rows and columns permutations, $G$ becomes $\tilde{G}$.  
Therefore, the last term in Eq. \eqref{eq:main_bound_proof2} can be reformulated as $\Vert L_yU_y-PAQ\tilde{G}\Vert=\Vert L_yU_y-PAGQ_y\Vert$ 
where $G$ is the random matrix in Algorithm \ref{alg:randomized_lu}.
By applying Lemmas \ref{lem:rrlu_approx_err} and \ref{lem:bhatia} to $\Vert L_yU_y-PAQ\tilde{G}\Vert$ we get
\begin{equation}
\begin{array}{lll}
\Vert L_yU_y-PAQ\tilde{G}\Vert & = & \Vert L_yU_y-PAGQ_y\Vert \\\\ 
& \le & (k(n-k)+1)\sigma_{k+1}(AG) \\\\
& \le & (k(n-k)+1)\Vert G \Vert \sigma_{k+1}(A).
\end{array}
\end{equation}
Lemma \ref{lem:rand_gauss} provides that $\Vert PAQ\tilde{G}F-PAQ \Vert \le \sqrt{2nl\beta^2\gamma^2+1}\sigma_{k+1}(A)$ 
and $\Vert F \Vert \le \sqrt{l}\beta$. By combining Lemmas \ref{lem:rand_gauss} and \ref{lem:bigval} we get
\begin{equation}
\Vert LU-PAQ \Vert \le \left(2\sqrt{2nl\beta^2\gamma^2+1}+2\sqrt{2nl}\beta\gamma \left( k(n-k)+1 \right)\right)\sigma_{k+1}(A),
\end{equation}
which completes the proof.
\end{proof}

\begin{rem} The error in Theorem \ref{trm:rand_lu_err} may appear large, especially for the case 
where $k \approx \frac{n}{2}$ and $n$ is large. 
Yet, we performed extensive numerical experiments showing that the actual error is much smaller 
when using Gaussian elimination with partial pivoting. 
Note that the error can decrease by increasing $k$. 
Numerical illustrations appear in section \ref{sec:numerical_results}.
\end{rem}

We now present an additional error bound that relies on \cite{litvak2005smallest}. Asymptotically, this is a tighter bound for large values of $n$ and $l$ since
it contains the term $\sqrt{\frac{n}{l}}$, which is smaller than the term $\sqrt{nl}$ in Theorem \ref{trm:rand_lu_err}. See also Remark \ref{rem:asymp_bound}.

\begin{theorem}
\label{trm:error_bound2}
Given a matrix $A$ of size $m \times n$, integers $k$ and $l$ such that $l>\left(1+\frac{1}{\ln k}\right)k$ and $a_2>0$.
By the application of Algorithm \ref{alg:randomized_lu} with $A,k$ and $l$ as its input parameters, the randomized LU decomposition satisfies
\begin{equation}
\Vert LU-PAQ \Vert \le \left( 2\sqrt{\frac{a_1^2 n}{c_1^2 l}+1}+\frac{2a_1\sqrt{n}}{c_1\sqrt{l}}\left(k(n-k)+1\right)\right)\sigma_{k+1}(A),
\end{equation}
with probability not less than $1-e^{-a_2 n}-e^{-c_2 l}$. 
The value of $c_1$ is given in Eq. \eqref{c1_const}, the value of $c_2$ is given in Eq. \eqref{c2_const} and $a_1$ is given by Theorem \ref{trm:subgauss}. 
$a_1,c_1$ and $c_2$ depend on $a_2$.
\end{theorem}

\begin{proof}
By using steps 5,6,7 and 8 in Algorithm \ref{alg:randomized_lu}, we get that
\begin{equation}
\Vert LU-PAQ \Vert=\Vert L_yL_y^\dagger PAQ-PAQ \Vert.
\end{equation}
Then, from Lemma \ref{lem:rand_lu_bound1}
\begin{equation}
\label{eq:large_L_bound}
\Vert L_yL_y^\dagger PAQ-PAQ \Vert \le 2\Vert PAQ\tilde{G}F-PAQ\Vert+2\Vert F \Vert \Vert L_yU_y-PAQ\tilde{G}\Vert.
\end{equation}
From Lemma \ref{lem:rand_gauss2} we get that 
\begin{equation}
\label{eq:large_L_bound2}
\Vert PAQ\tilde{G}F-PAQ \Vert \le \sqrt{\frac{a_1^2n}{c_1^2 l}+1}\sigma_{k+1}(A).
\end{equation}
By using the same argumentation given in Theorem \ref{trm:rand_lu_err}, we get
\begin{equation}
\label{eq:large_L_bound3}
\Vert L_yU_y-PAQ\tilde{G}\Vert=\Vert L_yU_y-PAGQ_y\Vert \le \left(k(n-k)+1\right)\Vert G \Vert\sigma_{k+1}(A)
\end{equation}
where $G$ is the matrix used in Algorithm \ref{alg:randomized_lu} Step 1.
By combining Eqs. \eqref{eq:large_L_bound}, \eqref{eq:large_L_bound2} and \eqref{eq:large_L_bound3}, and 
since $\Vert F \Vert \le \frac{1}{c_1\sqrt{l}}$ and $\Vert G \Vert \le a_1\sqrt{n}$ (see Lemma \ref{lem:rand_gauss2} 
and Theorem \ref{trm:subgauss}, respectively),
we get that
\begin{equation}
\label{eq:large_L_bound_final}
\Vert LU-PAQ \Vert \le 2\sqrt{\frac{a_1^2 n}{c_1^2 l}+1}\sigma_{k+1}(A)+\frac{2a_1\sqrt{n}}{c_1\sqrt{l}}\left(k(n-k)+1\right)\sigma_{k+1}(A).
\end{equation}
\end{proof}

\subsection{Analysis of the bound in Theorem \ref{trm:error_bound2}}
In this section, we analyze the bound in Eq. \eqref{eq:large_L_bound_final}. Although Eq. \eqref{eq:large_L_bound_final} bounds the randomized LU decomposition error, this bound can be modified  
to be used for other randomized algorithms. 
Randomized algorithms use a range approximation step that generates a smaller matrix than the original matrix that approximates the range of the original matrix. 
The range approximation enables us to compute the approximated decomposition using a smaller matrix while maintaining a bounded error. The obtained
error of a randomized algorithm depends on the quality of the range approximation step.
Formally, range approximation of a matrix $A$ can be accomplished by finding an orthogonal matrix $Q$ such that $\Vert QQ^*A-A \Vert$ is bounded. Hence, $Q^*A$ is a smaller matrix than $A$ that approximates the range of $A$.
A randomized algorithm, which finds such an orthogonal basis, appears in \cite{halko2011finding} and described in Algorithm \ref{alg:randomized_orth}. This procedure is used in its non-orthogonal form using LU decomposition in steps 1-3 in Algorithm \ref{alg:randomized_lu}. 

\begin{algorithm}[H]
	\caption{Randomized algorithm with orthogonal basis}
	\label{alg:randomized_orth}
	\textbf{Input:} $A$ matrix of size $m \times n$ to decompose, matrix rank $k$, $l\ge k$ number of columns to use.\\
	\textbf{Output:} Matrix $Q$ of size $m \times k$ such that $\Vert QQ^*A-A \Vert$ is bounded, and $Q^*Q=I$
	are orthogonal permutation matrices, $L$ and $U$ are lower and upper triangular matrices, respectively.
	\begin{algorithmic}[1]
		\STATE Create a matrix $G$ of size $n \times l$ whose entries are i.i.d. Gaussian random variables 
		with zero mean and unit standard deviation.
		\STATE $Y \gets AG$.
		\STATE Construct a matrix $U$ whose columns form an orthonormal basis for the range of $Y$ using SVD.
		\STATE Construct a matrix $Q$ by grouping the first $k$ vectors from $U$.
	\end{algorithmic}
\end{algorithm}

By estimating $\Vert QQ^*A-A \Vert$ in the same way as was done in Eq. \eqref{eq:large_L_bound_final} and by using Lemma \ref{lem:rand_lu_bound1} instead of estimating  $\Vert LL^\dagger A-A \Vert$  we get 
\begin{equation}
\label{eq:bound_for_svd}
\Vert QQ^*A-A \Vert \le 2\sqrt{\frac{a_1^2 n}{c_1^2 l}+1}\sigma_{k+1}(A)+\frac{2a_1\sqrt{n}}{c_1\sqrt{l}}\sigma_{k+1}(A)
\end{equation}
with probability not less than $1-e^{-a_2 n}-e^{-c_2 l}$. 
The value of $c_1$ is given in Eq. \eqref{c1_const}, the value of $c_2$ is given in Eq. \eqref{c2_const} and the value of $a_2$ is given in Theorem \ref{trm:subgauss}.
All of them depend on $a_2$.

Equation ~\ref{eq:bound_for_svd} provides an alternative bound to the randomized SVD algorithm. 
By neglecting constants and by analyzing the asymptotic behavior of Eq. \eqref{eq:bound_for_svd} we get that for $n \gg l$
\begin{equation}
\label{eq:bound_svd_asym}
\Vert QQ^*A-A \Vert \le 2\sqrt{\frac{a_1^2 n}{c_1^2 l}+1}\sigma_{k+1}(A)+\frac{2a_1\sqrt{n}}{c_1\sqrt{l}}\sigma_{k+1}(A) \propto \sqrt{\frac{n}{l}}\sigma_{k+1}(A)
\end{equation}
with an asymptotic failure probability of $e^{-c_2 l}$.
Two bounds are given in \cite{halko2011finding}: 1. Expectation-based bound and 2. probability-based bound. An expectation-based bound that is sharp appears in \cite{witten2013randomized} and the probability bound is better than previously developed bounds in \cite{randecomp}. This probability-based bound is given in \cite{halko2011finding}, Corollary 10.9:

\begin{corollary}[\cite{halko2011finding}]
For $Q$ from Algorithm \ref{alg:randomized_orth} and $p \ge 4$ ($p=l-k$),
\begin{equation}
\label{eq:bound_halko}
\Vert QQ^*A-A \Vert \le (1+17\sqrt{1+k/p})\sigma_{k+1}+\frac{8\sqrt{k+p}}{p+1}\left(\sum_{j>k} \sigma_j^2\right)^{1/2}
\end{equation}
with failure probability of at most $6e^{-p}$.
\end{corollary}
We now compare the asymptotic behavior of Eqs. \eqref{eq:bound_halko} with \eqref{eq:bound_for_svd} for the case  of a fixed $\sigma_j$, $j>k$,  $\sigma_{k+1}=\sigma_{k+2}=\dots=\sigma_{\min(m,n)}$. The asymptotic behavior for $n \gg k+p$ of Eq. \eqref{eq:bound_halko} is given by
\begin{equation}
\label{eq:bound_halko_asym}
\Vert QQ^*A-A \Vert \le (1+17\sqrt{1+k/p})\sigma_{k+1}+\frac{8\sqrt{k+p}}{p+1}\left(\sum_{j>k} \sigma_j^2\right)^{1/2} \propto \frac{\sqrt{(k+p)(n-k)}}{p+1}\sigma_{k+1}.
\end{equation}
Comparison between Eqs. \eqref{eq:bound_halko_asym} and \eqref{eq:bound_svd_asym} shows that Eq. \eqref{eq:bound_svd_asym} provides a better bound since Eq. \eqref{eq:bound_halko_asym} has an additional factor of $\sqrt{k+p}$ in the numerator in comparison to Eq. \eqref{eq:bound_svd_asym} and a smaller denominator than the one in Eq. \eqref{eq:bound_svd_asym}. Also, the failure probability is smaller in Eq. \eqref{eq:bound_svd_asym} since the exponents depend on $l$ instead of $p$.

The bound in Eq. \eqref{eq:bound_for_svd} is useful especially for large values of $l$. We assume that $n \gg l$ and $\sigma_j=\sigma$ for $j>k$. Next, we show a numerical example that illustrates the bounds. $m=2\cdot 10^8, n=10^8, k=990, l=1000$ and $a_2=1$. Computation of $a_1,c_1$ and $c_2$, which uses Theorems \ref{trm:subgauss} and \ref{trm:smallest_sv} provides $a_1=15.68$, $c_1=0.022$, $c_2=0.011$. Substituting these values in Eq. \eqref{eq:bound_for_svd}, privde
\begin{equation}
\label{eq:num_example1}
\Vert QQ^*A-A \Vert \le 2.9 \cdot 10^5 \sigma_{k+1}
\end{equation}
with failure probability  $1.1 \cdot 10^{-49}$. The same setup for Eq. \eqref{eq:bound_halko} gives
\begin{equation}
\label{eq:num_example2}
\Vert QQ^*A-A \Vert \le 7.28\cdot 10^5 \sigma_{k+1}
\end{equation}
with failure probability $2.72\cdot 10^{-4}$. Clearly, in this example, Eq.  \ref{eq:bound_for_svd} provides a better bound for both accuracy and failure probability.

Figure \ref{fig:bound_oversample} compares the asymptotic behaviors of the bounds in Eqs. \eqref{eq:bound_halko_asym} and \eqref{eq:bound_svd_asym}. This figure shows that when there is a small oversampling (small $p$), then the bound in Eq. \eqref{eq:bound_for_svd}, which is indicated by the red line, is asymptotically better in comparison to the bound in Eq. \eqref{eq:bound_halko_asym} which is indicated by the dashed blue line. As the oversampling increases, the bounds coincide.
\begin{figure}[H]
	\centering
	\includegraphics[scale=0.1]{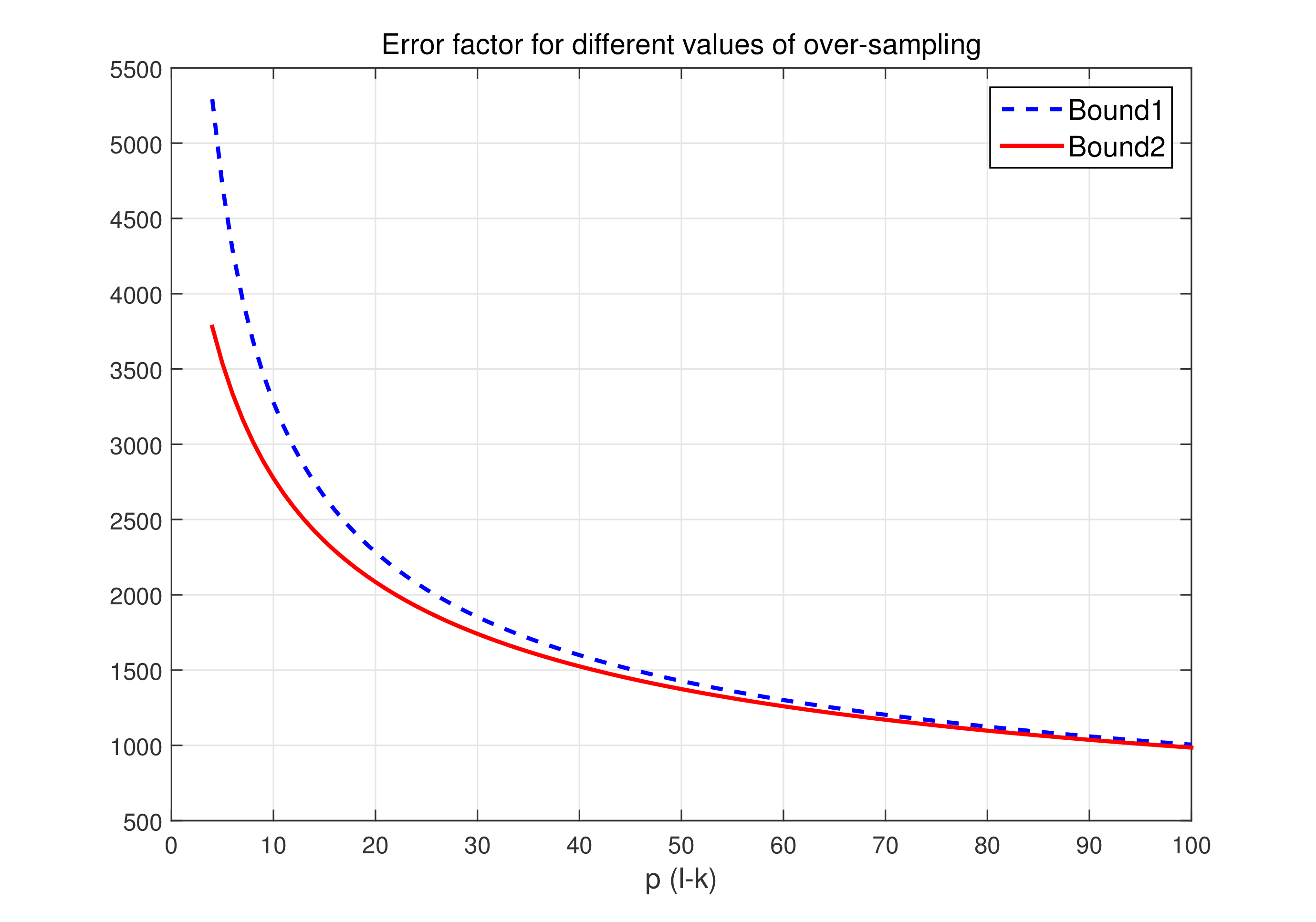}
	\caption{Bound values vs. oversampling for $k=3$, $p=4,5,\ldots,100$, $l=k+p$.}
	\label{fig:bound_oversample}
\end{figure}
Figure \ref{fig:bound_fixedp} shows the asymptotic behavior of Eq. \eqref{eq:bound_for_svd} for different values of $k$ and a fixed $p$. The red line illustrates Eqs. \eqref{eq:bound_svd_asym} and  \eqref{eq:bound_for_svd} and the blue dashed line illustrates Eqs. \eqref{eq:bound_halko_asym} and \eqref{eq:bound_halko}.
\begin{figure}[H]
	\centering
	\includegraphics[scale=0.1]{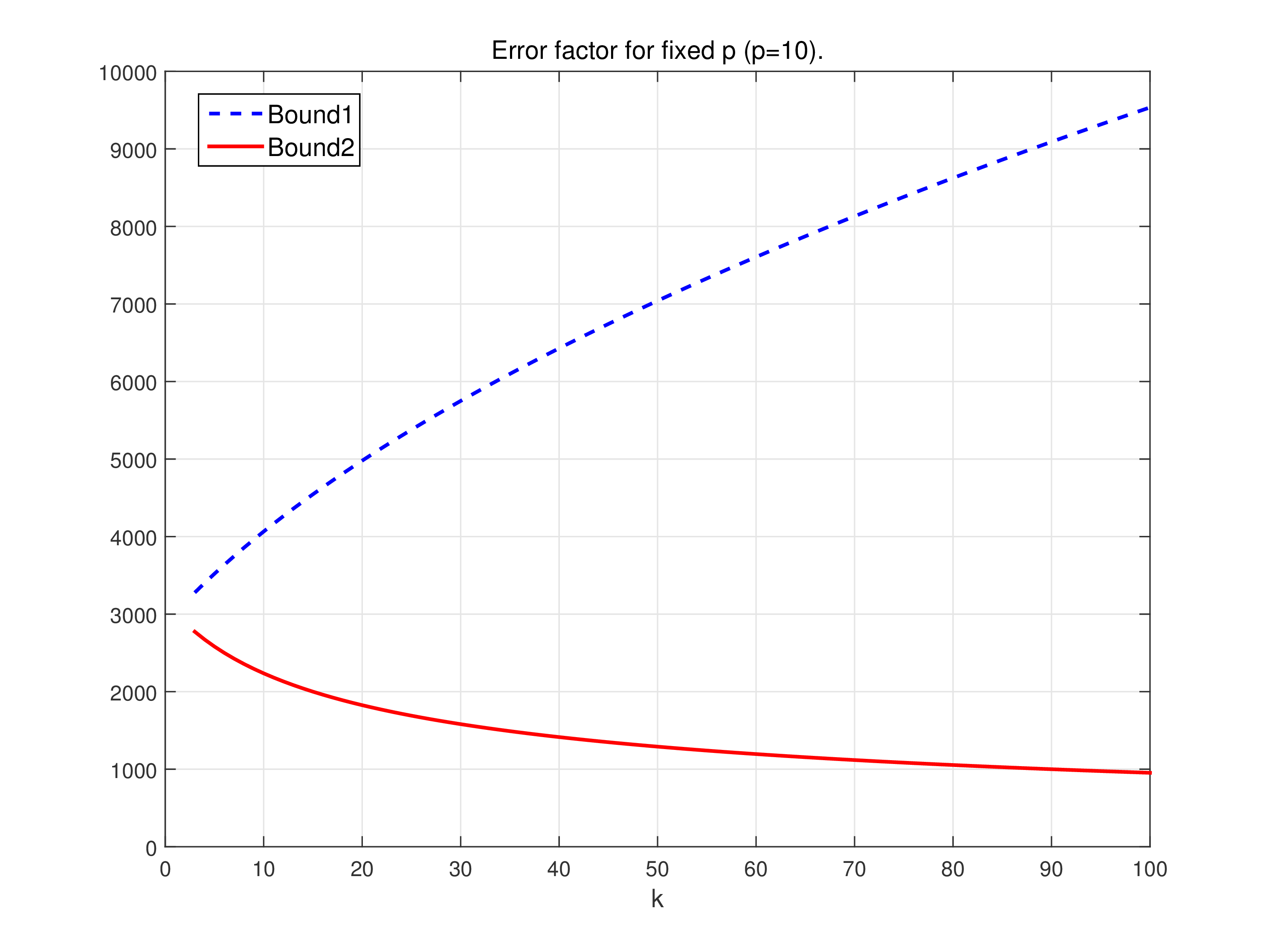}
	\caption{Bound values  for fixed $p=10$ where $k=3,4,\ldots,100$.}
	\label{fig:bound_fixedp}
\end{figure}

\subsection{Rank Deficient Least Squares}
In this section, we use the randomized LU 
to solve efficiently the Rank Deficient Least Squares (RDLS) problem. 
Assume that $A$ is an $m \times n$ matrix ($m \ge n$) with $\text{rank}(A)=k$, 
$k<n$ and $b$ is a column vector of size $m \times 1$. 
We want to minimize $\Vert Ax-b \Vert$. Because $A$ is a rank deficient matrix, then the problem 
has an infinite number of solutions. 
We show that the complexity of the solution depends on the rank of $A$ and that the problem is equivalent to 
solving the following two problems: a full rank Least Square (LS) problem of size $m \times k$ 
and a simplified undetermined linear system of 
equations that requires  a matrix inversion of size $k \times k$.
The solution is derived by the application of Algorithm \ref{alg:randomized_lu} to $A$ to get
\begin{equation}
\Vert Ax-b \Vert = \Vert P^T LU Q^T x-b \Vert = \Vert LUQ^Tx-Pb \Vert,
\end{equation}
where $L$ is an $m \times k$ matrix, $U$ is a $k \times n$ matrix and both $L$ and $U$ are of rank $k$. 
Let $y=UQ^Tx$ and $c=Pb$. Then, the problem is reformulated as $\min \Vert Ly-c \Vert$.
Note that $L$ is a full rank matrix and the problem to be solved becomes a standard full rank LS problem. 
The solution is given by $y=L^\dagger c$.
Next, we solve
\begin{equation}
\label{eq:rank_def_uzy}
Uz=y,
\end{equation}
where $z=Q^T x$. 
Since $U$ is a $k \times n$ matrix, Eq. \eqref{eq:rank_def_uzy}  is an underdetermined system. 
Assume that $U=\left[U_1 ~U_2\right]$ and $z=[z_1 ~z_2]^T$, 
where $U_1$ is a $k \times k$ matrix, $z_1$ is a $k \times 1$ vector and $z_2$ is a $(n-k) \times 1$ vector. 
Then, the solution is given by setting any value to $z_2$ and solving $U_1z_1=y-U_2 z_2.$
For simplicity, we choose $z_2=0$. Therefore, we get $z_1=U_1^{-1}y.$
The final solution is given by $x=Qz.$ 
This procedure is summarized in Algorithm \ref{alg:deficient_ls} 
that finds the solution to the deficient least squares problem that uses Algorithm \ref{alg:randomized_lu}.
\\

\begin{algorithm}[H]
\caption{Solving Rank Deficient Least Squares with Randomized LU}
\label{alg:deficient_ls}
\textbf{Input:} Matrix $A$ of size $m \times n$ with rank $k$, $l$,~$l\ge k$, $b$ vector of size $m \times 1$.\\ 
\textbf{Output:} Solution $x$ that minimizes $\Vert Ax-b \Vert$.
\begin{algorithmic}[1]
\STATE Apply Algorithm \ref{alg:randomized_lu} to $A$ with parameters $k$ and $l$.
\STATE $y \gets L^\dagger Pb$.
\STATE $z_1 \gets U_1^{-1} y$. 
\STATE $z \gets \binom{ z_1}{z_2 }$,  where $z_2 $ is an $n-k$ zero vector.
\STATE $x \gets Q z$.
\end{algorithmic}
\end{algorithm}
The complexity of Algorithm \ref{alg:deficient_ls} is equal to the randomized LU complexity (Algorithm \ref{alg:randomized_lu}) 
with an additional inversion cost of the matrix $U_1$ in Step 3, which is of size $k \times k$. 
Note that the solution given by Algorithm \ref{alg:deficient_ls} is sparse in the sense 
that $x$ contains at most $k$ non-zero entries.

\subsection{Fast Randomized LU}
\label{sec:fast_random_lu}
Algorithm \ref{alg:randomized_lu} describes
the randomized LU algorithm. This algorithm computes the LU
approximation of the matrix $A$  of rank $k$ whose computational
complexity is $\mathcal{C}_{RandLU}=\mathcal{O}(lmn + l^2 m + k^3 +
k^2 n)$ operations. We present now an asymptotic improvement to
Algorithm \ref{alg:randomized_lu} called fast randomized LU whose computational complexity is
\begin{equation}
\label{better_complexity}
\mathcal{C}_{FastRandLU}=\mathcal{O}(mn\log l + mkl + nkl + mk^2 +
k^3).
\end{equation}
In order to achieve it, we use the SRFT matrix and the ID Algorithm~\cite{low_rank_comp}, 
which were presented in sections \ref{subsec:srht} and \ref{subsec:id}, respectively.

The most computationally expensive procedures are steps 2 and 5 in
Algorithm \ref{alg:randomized_lu}. Step 2 involves matrix
multiplication with the matrix $A$ where  $A$ applied to a
random matrix.  Instead of projecting it with a Gaussian random matrix,
we use the SRFT matrix $R$. Due to the special structure $R=DFS$ (Section \ref{subsec:srht}), 
as was shown in Lemma \ref{lem:SRFTmult}, the
application of an $m \times n$ matrix $A$ to an $n \times l$
matrix $R$ necessitates $\mathcal{O}(nm\log l)$ floating point operations.

Instead of direct computation of $L_y^{\dagger}PA$ in step 5 in
Algorithm \ref{alg:randomized_lu}, $A$ is approximated by the ID of
$Y$, namely, if $Y=XY_{(J,:)}$ is the full rank ID of $Y$, then $A
\approx XA_{(J,:)}$.

\begin{algorithm}[H]
	\caption{Fast Randomized LU Decomposition}
	\label{alg:fast_randomized_lu}
	\textbf{Input:} Matrix $A$ of size $m \times n$ to decompose, $k$  desired rank, $l$ number of columns to use.\\
	\textbf{Output:} Matrices $P,Q,L,U$ such that $\Vert PAQ-LU \Vert \le \mathcal{O}(\sigma_{k+1}(A))$ where $P$ and $Q$
	are orthogonal permutation matrices, $L$ and $U$ are the lower and upper triangular matrices, respectively.
	\begin{algorithmic}[1]
		\STATE Create a random SRFT matrix $R$ of size $n \times l$ (Lemma \ref{lem:SRFTmult}).
		\STATE $Y \gets AR$.
		\STATE Apply RRLU decomposition to $Y$ such that $PYQ_y=L_yU_y$.
		\STATE Truncate $L_y$ and $U_y$ by choosing the first $k$ columns and the first $k$ rows, respectively, such that $L_y \leftarrow L_y(:,1:k)$ and $U_y \leftarrow U_y(1:k,:).$
		\STATE Compute the full rank ID decomposition of $Y$ such that $Y=XY_{(J,:)}$ (Section \ref{subsec:id}).
		\STATE $B \gets L_y^{\dagger}PXA_{(J,:)}$.
		\STATE Apply the LU decomposition to $B$ with column pivoting $BQ=L_bU_b$.
		\STATE $L \gets L_y L_b$.
		\STATE $U \gets U_b$.
	\end{algorithmic}
\end{algorithm}

\subsubsection{Computational complexity}
To compute the number of floating points operations in Algorithm \ref{alg:fast_randomized_lu}, we evaluate the complexity of each step:
\begin{description}
	\item[Step 1:] The multiplication of an $m \times n$ matrix $A$ by an $n \times l$ matrix $R$ requires $\mathcal{O}(nm\log l)$
	operations;
	\item[Step 2:] The computation of the RRLU decomposition of an $m \times l$ Y requires
	$\mathcal{O}(ml^2)$ operations;
	\item[Step 3:] The truncation of $L_y$ and $U_y$ requires $\mathcal{O}(1)$ operations;
	\item[Step 4:] The computation of the ID decomposition of $Y$ requires
	$\mathcal{O}(ml^2)$ operations;
	\item[Step 5:] The computation of the pseudo inverse $L_y^{\dagger}$ requires
	$\mathcal{O}(k^2m+k^3)$ operations;
	\item [Step 6:] The multiplication of $L_y^{\dagger}PXA_{(J,:)}$ requires $\mathcal{O}(mkl +
	nkl)$ operations;
	\item[Step 7:] The computation of the partial pivoting LU of matrix $B$ requires
	$\mathcal{O}(nk^2)$ operations;
	\item[Step 8:] The computation of $m \times k$ matrix $L$ requires $\mathcal{O}(mk^2)$ operations.
\end{description}

The total computational complexity of Algorithm \ref{alg:fast_randomized_lu} is
$\mathcal{O}(mn\log l + mkl + nkl + mk^2 + k^3)$. By simplifying
this expression while assuming that $k$ and $l$ are of the same
magnitude, we get that the total computational complexity of Algorithm \ref{alg:fast_randomized_lu} is
$\mathcal{O}(mn\log k + (m+n)k^2 + k^3)$.

\subsubsection{Correctness Algorithm \ref{alg:fast_randomized_lu}}
We now prove that Algorithm \ref{alg:fast_randomized_lu} 
approximates the LU decomposition and provide an error
bound.
\begin{theorem}
	\label{trm:fastRand_lu_err} Given a matrix $A$ of size $m\times n.$
	Its fast randomized LU decomposition in Algorithm
	\ref{alg:fast_randomized_lu} with integers $k$ and $l$  (where $n,m \geq l
	\ge k$ sufficiently large)  satisfies
	
	\begin{equation*}
		\begin{array} {llll}
			\Vert LU-PAQ \Vert  & \le && \left( \left[ 1 +\sqrt{1 + 4k(n - k)}\right]\sqrt{1+7n/l}\right)\sigma_{k+1}(A) \\
			&       &+& 2 \left( \sqrt{\alpha   n + 1}+ \sqrt{\frac{\alpha }{l}} (k(n-k)+1) \right)\sigma_{k+1}(A)
		\end{array}
	\end{equation*}
	with probability not less than $1-3\frac{1}{\beta k}$ where
	$\beta>1$ is a constant.
\end{theorem}

The proof of Theorem \ref{trm:fastRand_lu_err} uses Lemmas
\ref{lem:SRFTerrbnd}-\ref{lem:ARTminAbound}.

\begin{lemma} \label{lem:SRFTerrbnd}
	Let $A$ be an $m \times n$ matrix with singular values $\sigma_1 \geq \sigma_2 \geq \ldots \geq \sigma_{\min(m,n)}$. Let $k$ and $l$ be integers
	such that $4\left [ \sqrt{k} + \sqrt{8\ln(kn)} \right ] ^2 \ln k \leq l \leq n.$
	Let $R$ be an $ n \times l$ SRFT matrix and $Y=AR$. Denote by $Q$ the $m \times l$ matrix whose columns form an orthonormal
	basis for the range of $Y$.
	Then, with a failure probability of at most $3k^{-1}$ we have
	$$
	\|A-QQ^*A\| \leq \sqrt{1+7n/l} \sigma_{k+1}.
	$$
\end{lemma}
Lemma \ref{lem:SRFTerrbnd} appears in \cite{halko2011finding} as Theorem 11.2 and in a
\cite{tropp2011improved} as Theorem 3.1 in slightly different formulation.

\begin{lemma}
	\label{lem:ID_approx} Let $A$ be an $m  \times n$ matrix,  $R$ is
	an $n \times l$ SRFT random matrix and  $Y = XY_{(J,:)}$ is the
	full rank ID of $Y = AR$. Then,
	$$
	\Vert  A - XA_{(J,:)} \Vert \leq \left( 1 +\sqrt{1 + 4k(n - k)}\right)\sqrt{1+7n/l} \sigma_{k+1}(A)
	$$
	with failure probability of at most $ 3k^{-1}$
	when  $l \geq 4 \left(\sqrt{k} + \sqrt{8\ln(kn)}\right)^2\ln k$.
\end{lemma}
The proof is the same as in Lemma 5.1 in \cite{halko2011finding}.
\begin{proof}
	Denote by $Q$ the matrix whose columns form an orthonormal basis for
	the range of Y. By using Lemma \ref{lem:SRFTerrbnd} we have
	\begin{equation}\label{eq:AQQA}
		\Vert A-QQ^*A \Vert \leq \sqrt{1+7n/l} \sigma_{k+1}(A)
	\end{equation}
	except with probability $3k^{-1}$. 
	
	Denote $\hat{A}=QQ^*A$. Since $\hat{A} = XQ_{(J,:)}Q^*A$ and $X_{(J,:)}= I$, we have $\hat{A}_{(J,:)} = Q_{(J,:)}Q^*A$. Thus, $\hat{A}=X\hat{A}_{(J,:)}$.

	\begin{equation*}
	\begin{array}{lll}
		\|  A - XA_{(J,:)} \| &=& \| A - X\hat{A}_{(J,:)} + X\hat{A}_{(J,:)} - XA_{(J,:)} \| \\
		& \leq & \|A-\hat{A}\| + \|X\hat{A}_{(J,:)} - XA_{(J,:)} \| \\
		&   =  & \|A-\hat{A}\| + \|X\|\|\hat{A}_{(J,:)} - A_{(J,:)} \| \\
		& \leq & (1+\|X\|) \|A-\hat{A}\|.
	\end{array}
	\end{equation*}
	By using Eq. \eqref{eq:AQQA} we have
	$$
	\|  A - XA_{(J,:)} \| \leq (1+\|X\|) \sqrt{1+7n/l} \sigma_{k+1}(A).
	$$
	The proof is completed since $X$ contains a $k \times k$ identity matrix and the spectral norm of the remaining $(n-k) \times k$ submatrix is bounded by 2.

\end{proof}

\begin{lemma}[Appears in \cite{WLRT} as Lemma 4.6]
	\label{lem:ARTminAbound} Suppose that $k,l,n $ and $m$ are positive
	integers with $k \leq l$ such that $l< \min{(m,n)}$. Suppose 
	that $\alpha$ and $\beta$ are real numbers greater than 1 such that
	$$
	m>l \geq \frac{\alpha^2\beta}{(\alpha - 1)^2}k^2.
	$$
	Suppose that $A$ is an $m \times n$ complex matrix and $Q$ is the $n
	\times l$ SRFT matrix. Then, there exists an $l \times n$ complex
	matrix F such that $ \Vert AQF -A \Vert \leq \sqrt{\alpha n + 1}
	\sigma_{k+1} $ and $ \Vert F \Vert \leq \sqrt{\frac{\alpha}{l}} $
	with probability at least $1-\frac{1}{\beta}$ where $\sigma_{k+1}$
	is the $(k+1)$th greatest singular value of $A$.
\end{lemma}

\begin{lemma} \label{lem:srftLUbound}
	Let $A,P,Q, L_y$ and $ L_y^{\dagger}$ be as in Algorithm
	\ref{alg:fast_randomized_lu}, then
	$$
	\Vert  L_y L_y^{\dagger} PAQ - PAQ \Vert \le  2 \left( \sqrt{\alpha   n + 1}+ \sqrt{\frac{\alpha }{l}} (k(n-k)+1) \right)\sigma_{k+1}(A)
	$$
	with probability of at least $1-\frac{1}{\beta}$ where $m>l \geq
	\frac{\alpha^2\beta}{(\alpha - 1)^2}k^2$.
\end{lemma}

\begin{proof}
	By applying Lemma \ref{lem:rand_lu_bound1} we get
	\begin{equation}
	\label{ineq1}
	\Vert L_yL_y^\dagger PAQ-PAQ \Vert  \le  2\Vert PAQ\tilde{G}F-PAQ\Vert+2\Vert         F \Vert \Vert L_yU_y-PAQ\tilde{G}\Vert.
	\end{equation}
	$U_y$ is the $k \times n$ matrix in step 4 in Algorithm
	\ref{alg:fast_randomized_lu}. This holds for any matrix $\tilde{G}$.
	In particular, for a matrix $\tilde{G}$ satisfies
	$Q\tilde{G}=RQ_y$, where $R$ is the SRFT matrix in Algorithm
	\ref{alg:fast_randomized_lu}. Therefore, the last term in Eq.
	\ref{ineq1} can be reformulated by $\Vert
	L_yU_y-PAQ\tilde{G}\Vert=\Vert L_yU_y-PARQ_y\Vert$. By applying
	Lemmas \ref{lem:rrlu_approx_err} and \ref{lem:bhatia} to $\Vert
	L_yU_y-PAQ\tilde{G}\Vert$ we get
	\begin{equation}
	\begin{array}{lll}
	\Vert L_yU_y-PAQ\tilde{G}\Vert & = & \Vert L_yU_y-PARQ_y\Vert \\\\
	& \le & (k(n-k)+1)\sigma_{k+1}(AR) \\\\
	& \le & (k(n-k)+1)\Vert R \Vert \sigma_{k+1}(A).
	\end{array}
	\end{equation}
	Since $R$ is an SRFT matrix, it is orthogonal, thus $\Vert R
	\Vert = 1$. Lemma \ref{lem:ARTminAbound} proves that $\Vert
	PAQ\tilde{G}F-PAQ \Vert \le  \sqrt{\alpha n + 1}\sigma_{k+1}(A)$ and
	$\Vert F \Vert \leq \sqrt{\frac{\alpha }{l}} $. By summing up, we
	get
	\begin{equation*}
		\begin{array}{lll}
			\Vert L_yL_y^\dagger PAQ-PAQ \Vert & \le & 2\Vert PAQ\tilde{G}F-PAQ\Vert+2\Vert         F \Vert \Vert L_yU_y-PAQ\tilde{G}\Vert \\
			
			&\le &  2 \sqrt{\alpha n + 1}\sigma_{k+1}(A)+ 2 \Vert F \Vert (k(n-k)+1)                 \sigma_{k+1}(A) \\
			& \le & 2 \left( \sqrt{\alpha n + 1}+ \sqrt{\frac{\alpha }{l}} (k(n-k)+1)
			\right)\sigma_{k+1}(A).
		\end{array}
	\end{equation*}
	
\end{proof}

Proof of Theorem \ref{trm:fastRand_lu_err}.
\begin{proof}
	By substituting $L$ and $U$ from Algorithm
	\ref{alg:fast_randomized_lu} we have
	\begin{equation}
	\label{ineq2}
	\begin{array}{lcl}
	\Vert LU - PAQ  \Vert & = & \Vert L_y L_b U_b - PAQ \Vert = \Vert L_yBQ - PAQ \Vert = \\
	& = & \Vert L_y L_y^{\dagger} PXA_{(J,:)}Q - PAQ \Vert  \leq \\
	& \leq & \Vert L_y L_y^{\dagger} PXA_{(J,:)}Q - L_y L_y^{\dagger} PAQ \Vert  + \Vert  L_y L_y^{\dagger} PAQ - PAQ \Vert.
	\end{array}
	\end{equation}	
The first term in the last inequality in Eq. \eqref{ineq2} is bounded in
the following way:
	$$
	\Vert L_y L_y^{\dagger} PXA_{(J,:)}Q - L_y L_y^{\dagger} PAQ \Vert  \leq \Vert L_y L_y^{\dagger} P \Vert \Vert XA_{(J,:)} - A \Vert  \Vert  Q\Vert   =     \Vert XA_{(J,:)} - A
	\Vert .
	$$
	By using Lemma \ref{lem:ID_approx} we get
	\begin{equation*}
		\Vert LU - PAQ  \Vert = \Vert L_y L_b U_b - PAQ \Vert \le \left( 1 +\sqrt{1 + 4k(n - k)}\right)\sqrt{1+7n/l}\sigma_{k+1}(A)
	\end{equation*}
	with probability of not less than $1 - 3k^{-1}$.
	
	The second term $\Vert  L_y L_y^{\dagger} PAQ - PAQ \Vert$ in the
	last inequality of Eq. \eqref{ineq2} is bounded by Lemma
	\ref{lem:srftLUbound}.
	
	By combining these results we get
	\begin{equation*}
		\begin{array} {llll}
			\Vert LU-PAQ \Vert  & \le && \left( \left[ 1 +\sqrt{1 + 4k(n - k)}\right]\sqrt{1+7n/l}\right)\sigma_{k+1}(A) \\
			&       &+& 2 \left( \sqrt{\alpha   n + 1}+ \sqrt{\frac{\alpha }{l}} (k(n-k)+1) \right)\sigma_{k+1}(A)
		\end{array}
	\end{equation*}
	which completes the proof.
\end{proof}

\section{Numerical Results}
\label{sec:numerical_results}
In order to evaluate Algorithm \ref{alg:randomized_lu},
we present the numerical results by comparing the performances of several randomized low rank approximation algorithms.
We tested the algorithms and compared them by applying them to random matrices and to images. 
All the results were computed using the standard MATLAB libraries including MATLAB's GPU interface
on a machine with two Intel Xeon CPUs X5560 2.8GHz that contains an nVidia GPU GTX TITAN card.

\subsection{Error Rate and Computational Time Comparisons}
\label{sub:error1}
The performance of the randomized LU (Algorithm \ref{alg:randomized_lu})  was tested and compared to the 
randomized SVD and to the randomized ID (see \cite{randecomp,halko2011finding}).
The tests compare the normalized (relative) error of the low rank approximation 
obtained by the examined methods. In addition, the computational time of each method was measured. 
If $A$ is the original matrix and $\hat{A}$ is a low rank approximation of $A$, 
then the relative approximation error  is given by:
\begin{equation}
\label{eq:relErr}
\text{err} = \frac{\Vert A-\hat{A} \Vert}{\Vert A \Vert}.
\end{equation}
We compared  the low rank approximation achieved by the application of the randomized SVD, 
randomized ID and randomized LU with different ranks $k$.
Throughout the experiments, we chose $l=k+3$ and the test matrix was a random 
matrix of size $3000 \times 3000$ with exponentially decaying singular values. 
The computations of the algorithms were done in a single precision.
The comparison results are presented in Fig. \ref{fig:random_error_single}.
The experiment shows that the error of the randomized ID is  larger than 
the error obtained from both the randomized SVD and the randomized LU (Algorithm \ref{alg:randomized_lu}), which are almost identical. 
In addition, we compared  the execution time of these algorithms.
The results are presented in Fig. \ref{fig:random_time_single_eight_cores}.
The results show that the execution time of the randomized LU (Algorithm \ref{alg:randomized_lu}) is  lower than the 
execution time of the randomized SVD and the randomized ID algorithms.
The LU factorization has a parallel implementation (see \cite{golub2012matrix} section 3.6). 
To see the impact of the parallel LU decomposition implementation, 
the execution time to compute the randomized LU of a matrix of size $3000 \times 3000$ was 
measured on an nVidia GTX TITAN GPU device and it is shown in Fig. \ref{fig:random_time_single_gpu}.
The execution time on the GPU was $10$ times faster than running it on an eight cores CPU. 
Thus, the algorithm scales well. 
For larger matrices ($n$ and $k$ are large), the differences between the performances while running on CPU and on GPU are more significant.

\begin{figure}[H]
	\centering
		\includegraphics[scale=0.1]{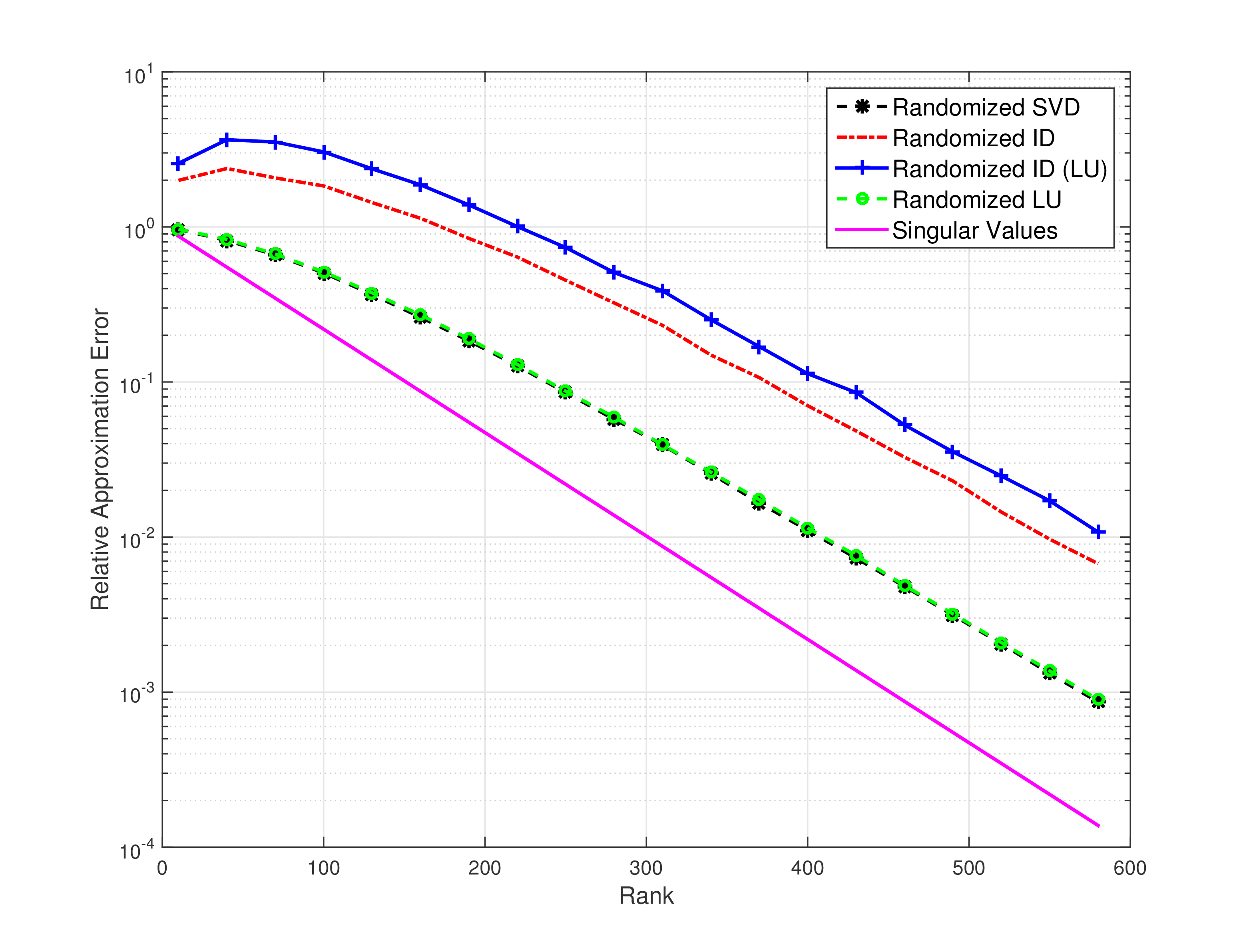}
	\caption{Low rank approximation error of different algorithms: Randomized SVD, Randomized ID (QR and LU) and Randomized LU with respect to the real singular values of the testing matrix.}
	\label{fig:random_error_single}
\end{figure}

\begin{figure}[H]
	\centering
		\includegraphics[scale=0.1]{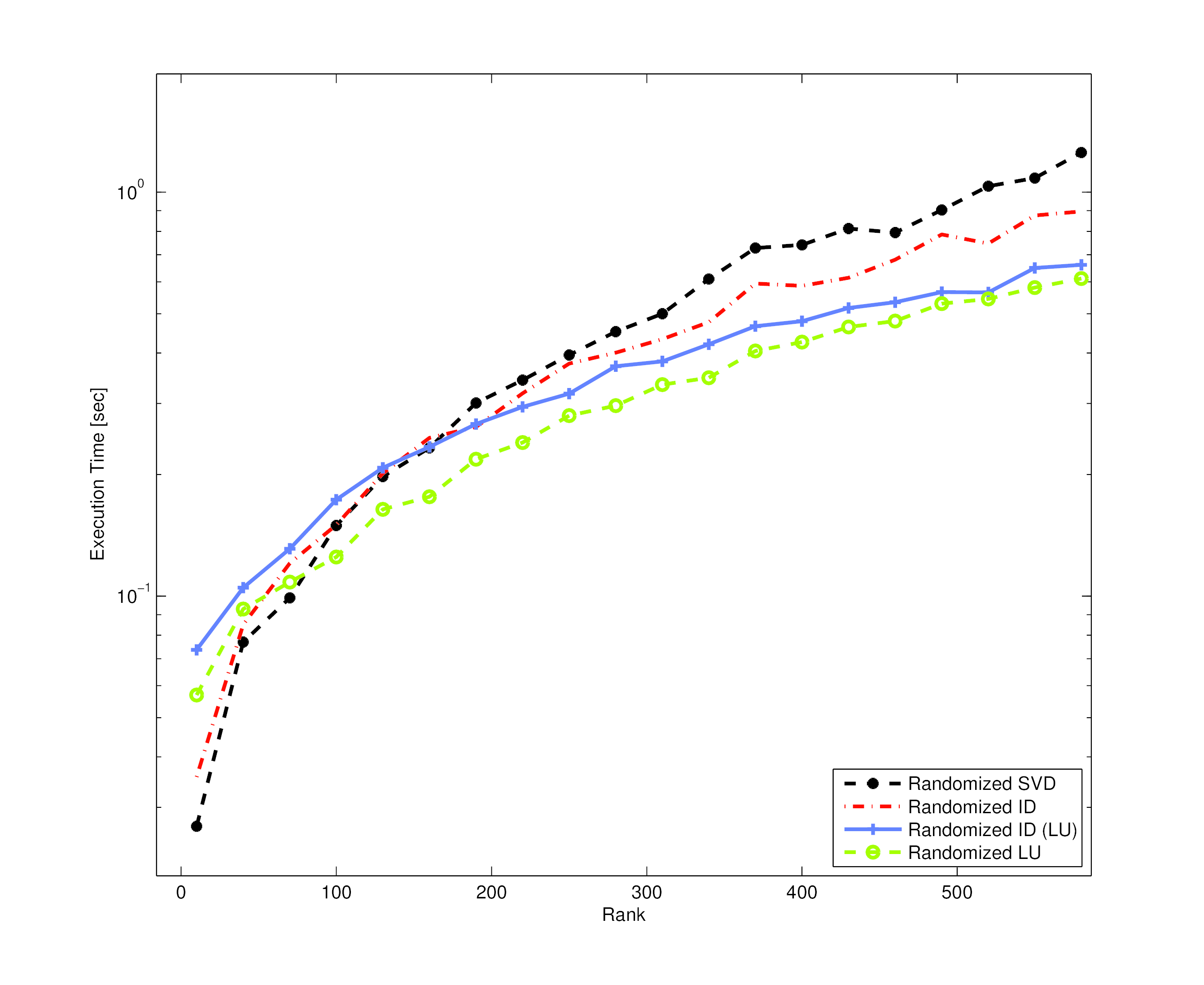}
	\caption{The execution times of the same algorithms as in Fig. \ref{fig:random_error_single} running on a CPU.}
	\label{fig:random_time_single_eight_cores}
\end{figure}

\begin{figure}[H]
	\centering
		\includegraphics[scale=0.1]{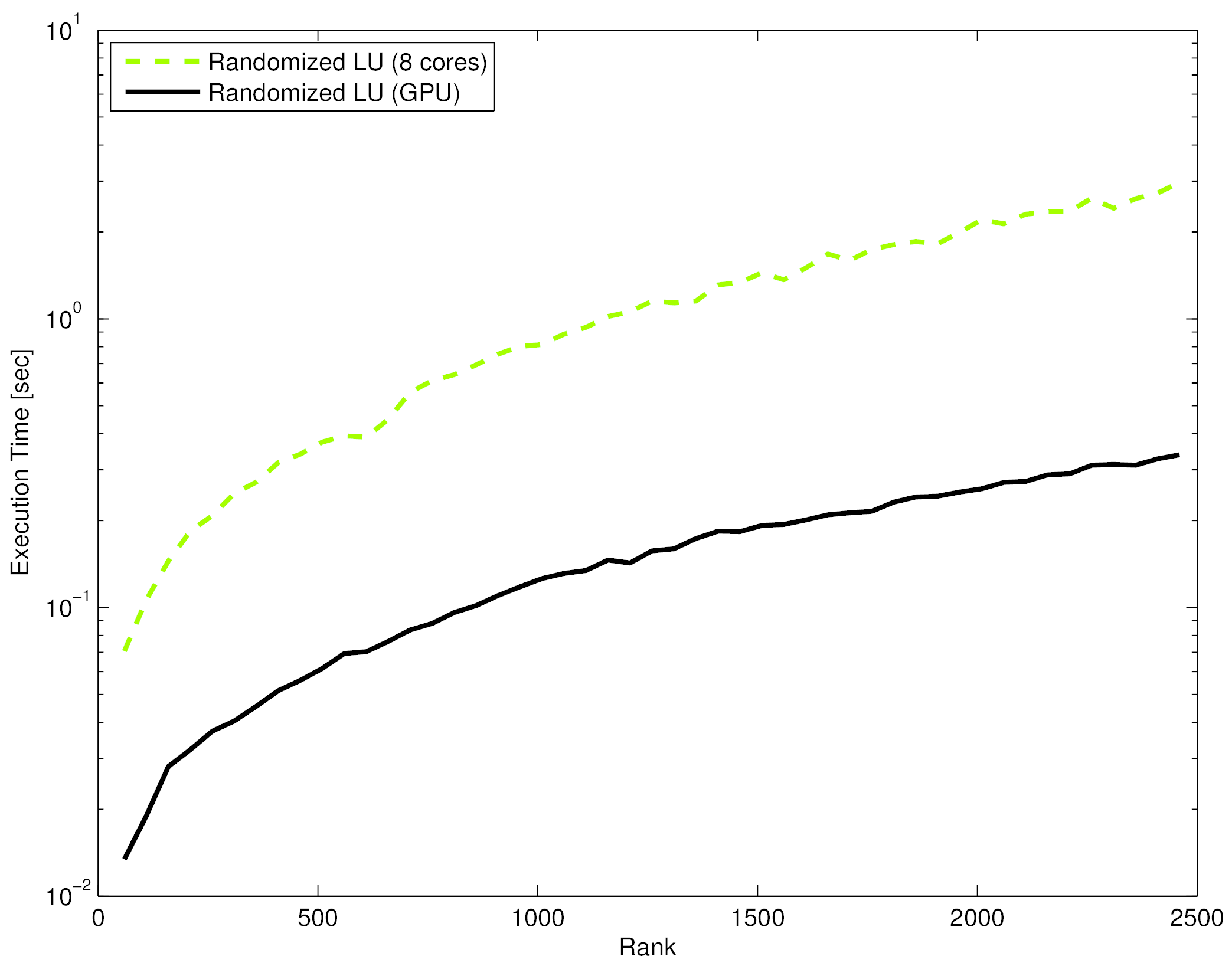}
	\caption{The execution times from running Algorithm \ref{alg:randomized_lu} on different 
computational platforms: CPU with 8 cores and GPU.}
	\label{fig:random_time_single_gpu}
\end{figure}

\subsection{Power Iterations}
The performance (error wise) of Algorithm \ref{alg:randomized_lu} can be further improved by the application of power iterations to the input matrix. Specifically, by replacing the projection step $Y \leftarrow AG$ with 
$Y\leftarrow (AA^*)^qAG$ for small integer $q$ (for larger values, a normalized scheme has to be used -- see \cite{martinsson2010normalized}). Applying the power iterations scheme to an exponentially decaying singular values, an improvement is achieved even for $q=1$. The same random matrix, which was described in section \ref{sub:error1}, is tested again with and without power iterations on both GPU and CPU. This time, double precision is used.

\begin{figure}[H]
	\centering
	\includegraphics[scale=0.1]{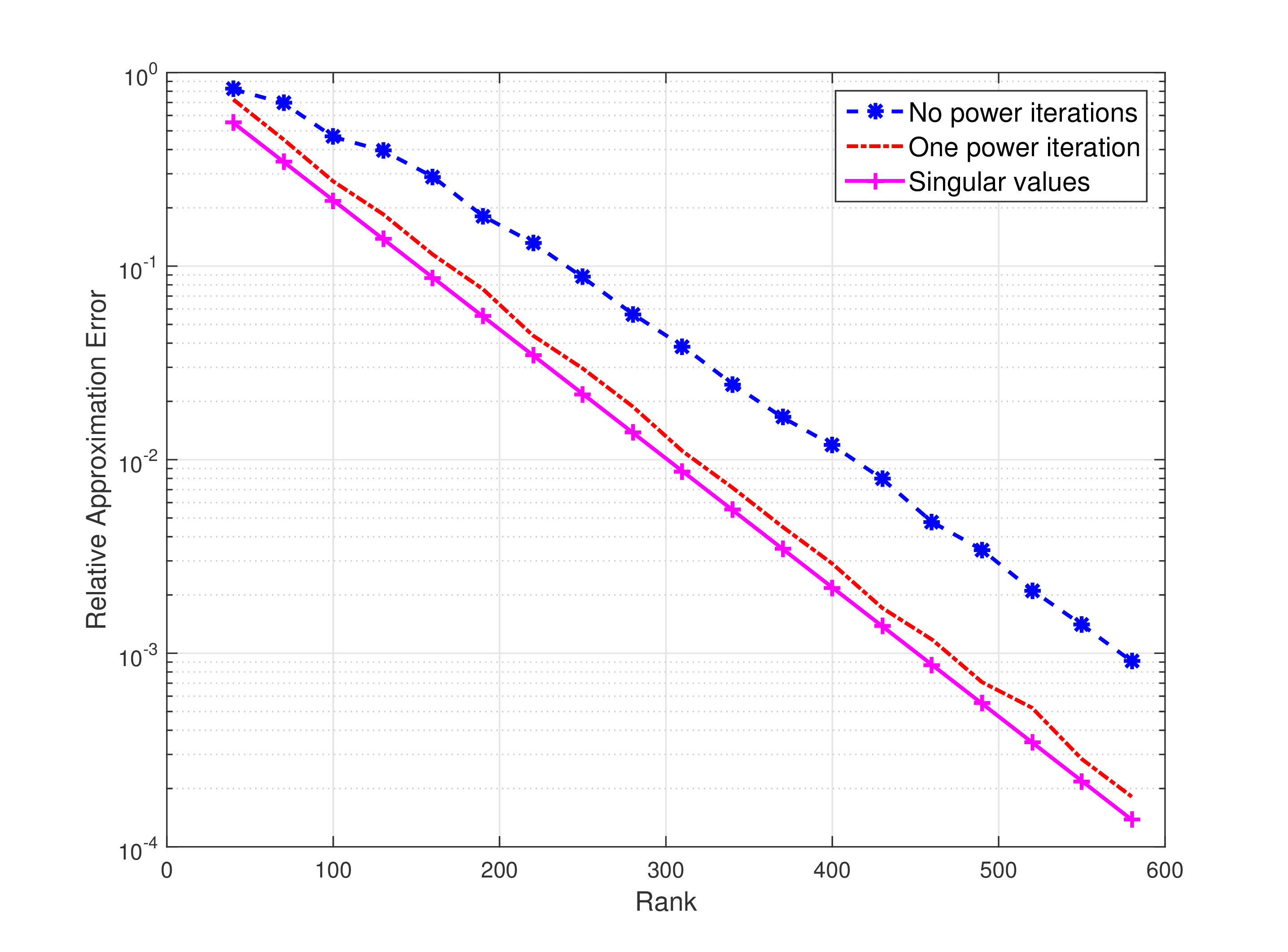}
	\caption{Low rank approximation using randomized LU decomposition (Algorithm \ref{alg:randomized_lu}) with and without power iterations.}
	\label{fig:power_iter_err}
\end{figure}

\begin{figure}[H]
	\centering
	\includegraphics[scale=0.13]{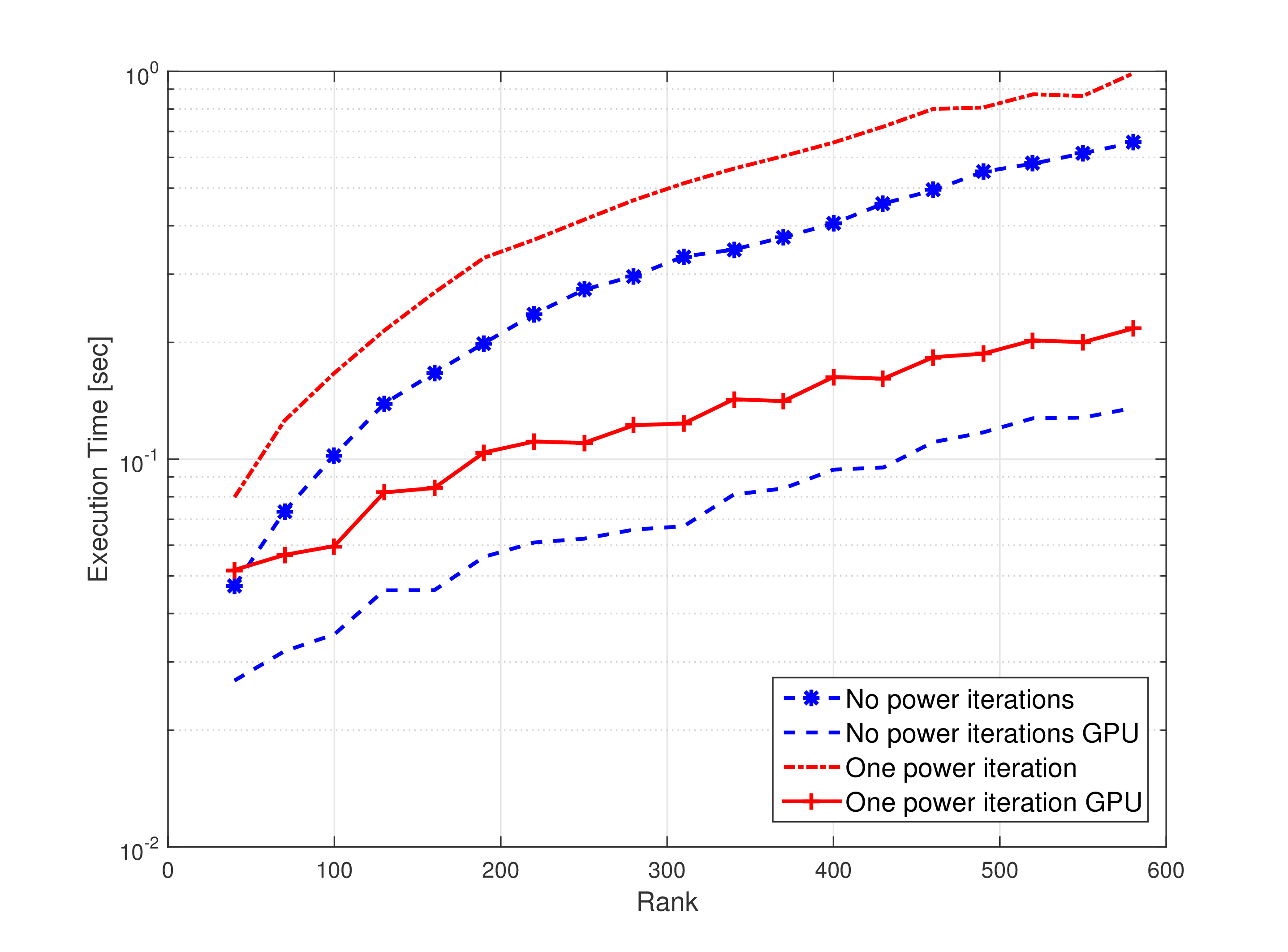}
	\caption{The execution time on CPU and GPU of the randomized LU with and without power iterations.}
	\label{fig:power_iter_time}
\end{figure}

The experiment was repeated for a slowly decaying singular values. The decay of the singular values was proportional to $1/k^2$. To make the decay slower than $1/k^2$ from $k=1$, a factor was added such that $\sigma_k=\frac{100}{(9+k)^2}$. The singular values were normalized such that $\sigma_1=1$. The results are shown in Fig. \ref{fig:power_iter_slow}.

\begin{figure}[H]
	\centering
	\includegraphics[scale=0.1]{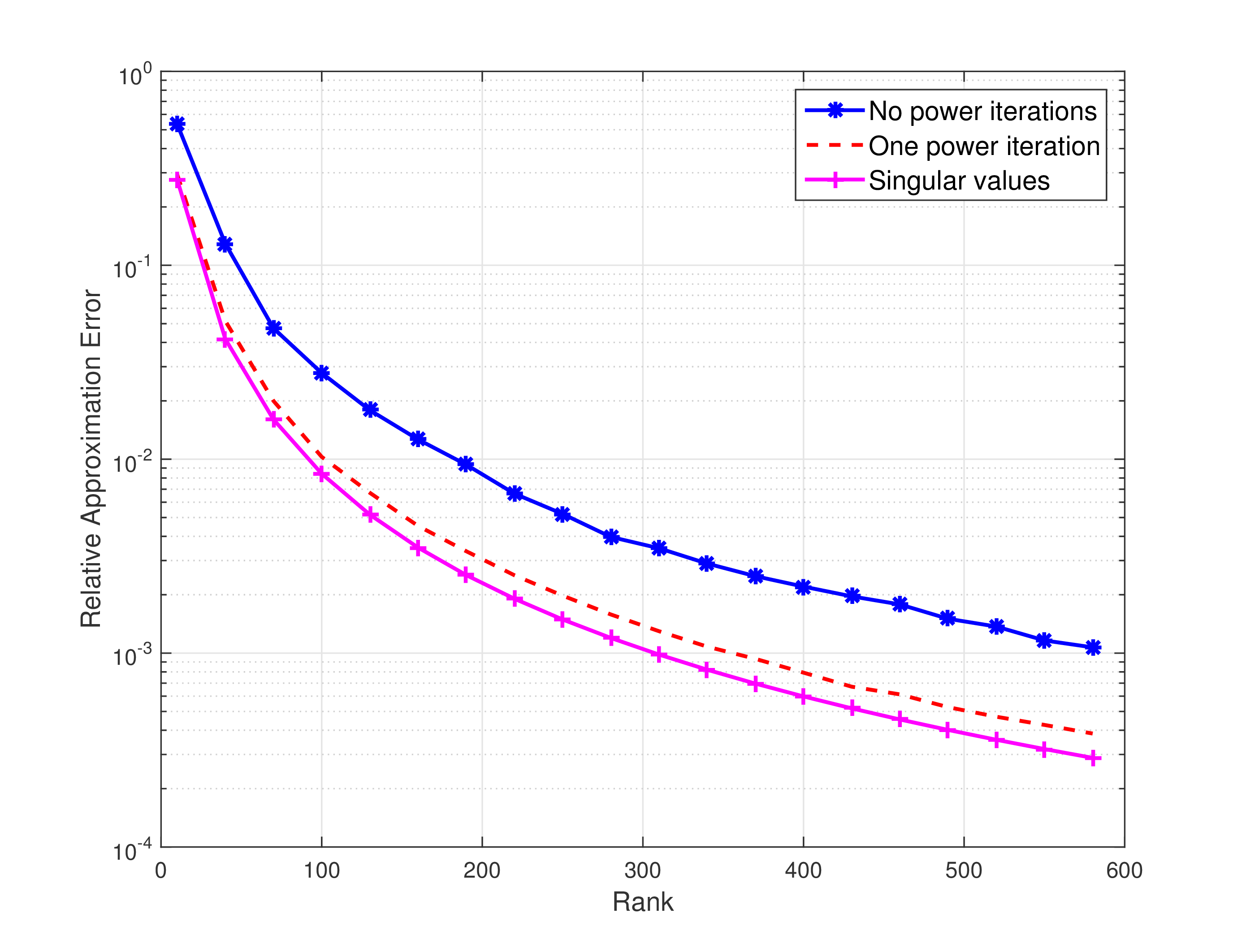}
	\caption{Low rank approximation using randomized LU decomposition with and without power iterations for slowly decaying singular values.}
	\label{fig:power_iter_slow}
\end{figure}

\subsection{Image Matrix Factorization}
Algorithm \ref{alg:randomized_lu} was applied to images given in a matrix form. 
The factorization error and the execution time were compared to the performances of the randomized SVD and to the randomized ID. 
We also added the SVD error and execution time computed by the Lanczos bidiagonalization \cite{golub2012matrix} that is implemented in the PROPACK package \cite{propack}.
The image size was $ 2124 \times 7225$ pixels and it has $256$ gray levels. The parameters were
$k=200$ and $ l=203$. The approximation quality (error) was measured in PSNR defined by
\begin{equation}
\text{PSNR}=20\log_{10}\frac{\max_A\sqrt{N}}{\Vert A-\hat{A}\Vert_F}
\end{equation}
where $A$ is the original image, $\hat{A}$ is the approximated image (the output from Algorithm \ref{alg:randomized_lu}), $\max_A$ is the maximal pixel 
value of $A$, $N$ is the total number of pixels and $\Vert \cdot \Vert_F$ is the Frobenius norm.

\begin{figure}[H]
	\centering
		\includegraphics[width=1.0\textwidth]{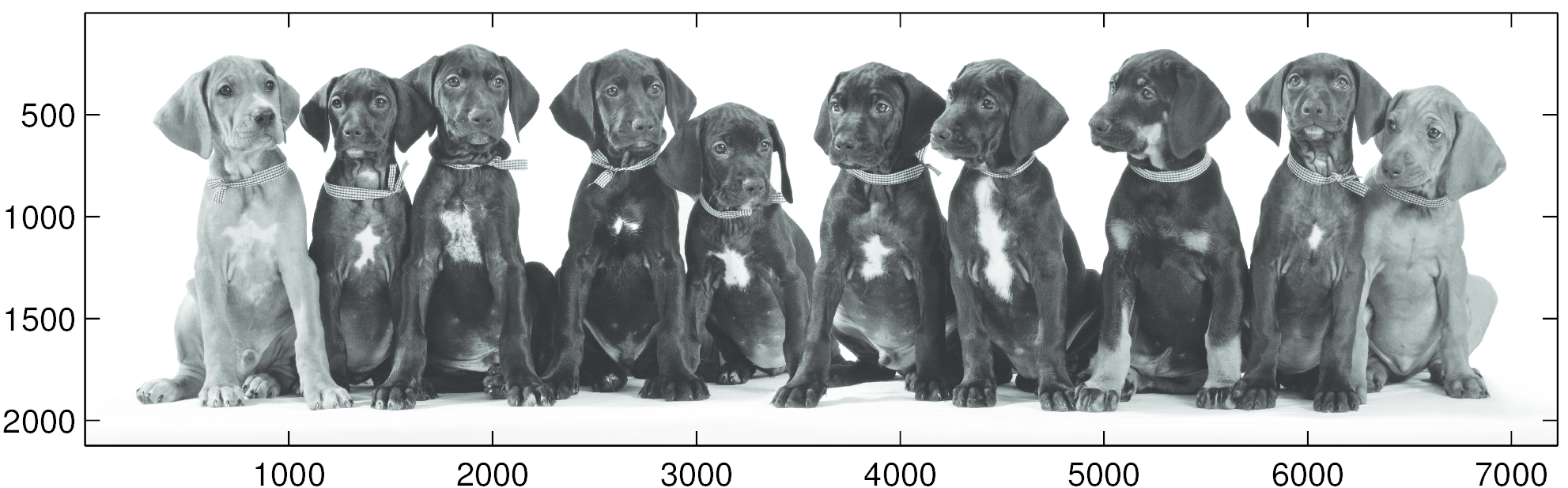}
	\caption{The original input image of size $2124 \times 7225$ that was factorized by the application of the randomized LU, randomized ID and randomized SVD algorithms.}
	\label{fig:puppies_image}
\end{figure}
\begin{figure}[H]
	\centering
		\includegraphics[width=1.0\textwidth]{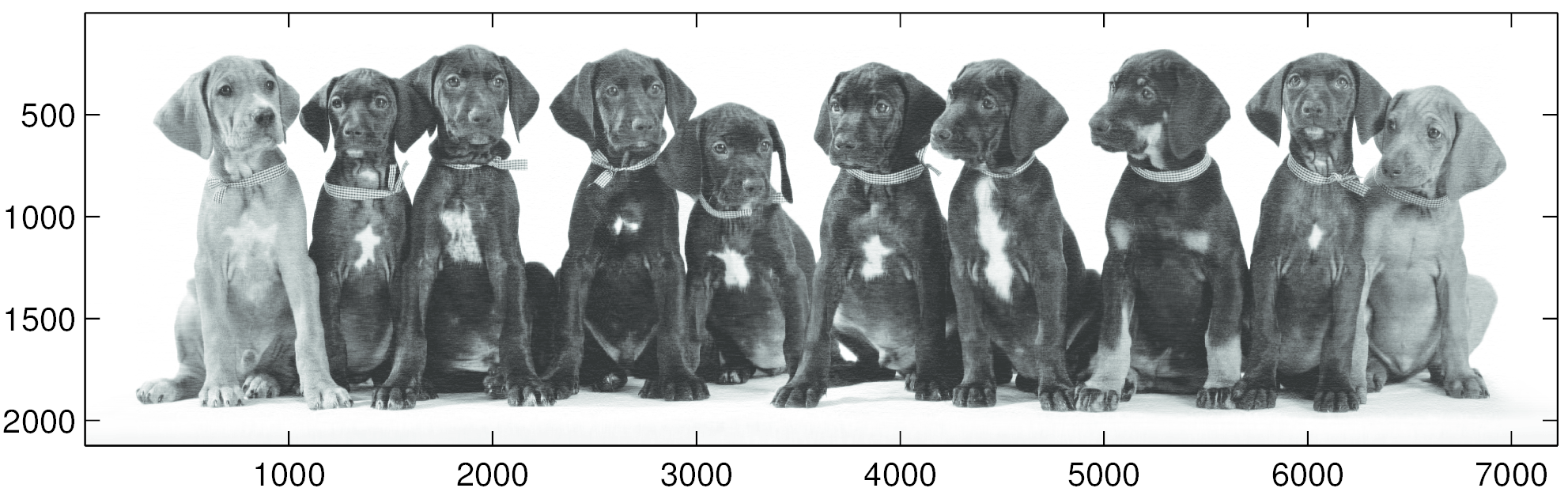}
	\caption{The reconstructed image from the application of the randomized LU factorization with $k=200$ and $l=203$.}
	\label{fig:puppies_image_lu}
\end{figure}

Figures \ref{fig:puppies_image} and \ref{fig:puppies_image_lu} show the original and the reconstructed images, respectively.
The image reconstruction quality (measured in PSNR) related to rank $k$ is shown in Fig. \ref{fig:puppies_error} where for the same $k$, the PSNR from the application of 
Algorithm \ref{alg:randomized_lu} is higher than the PSNR generated by the application of the randomized ID and almost identical to the randomized SVD. 
In addition, the PSNR values are close to the result achieved by the application of the Lanczos SVD which is the best possible rank $k$ approximation.
The execution time of each algorithm is shown in Fig. \ref{fig:puppies_time}. All the computations were done in double precision.
Here, the randomized LU is faster than all the other compared methods making it applicable for real time applications.

\begin{figure}[H]
	\centering
		\includegraphics[scale=0.1]{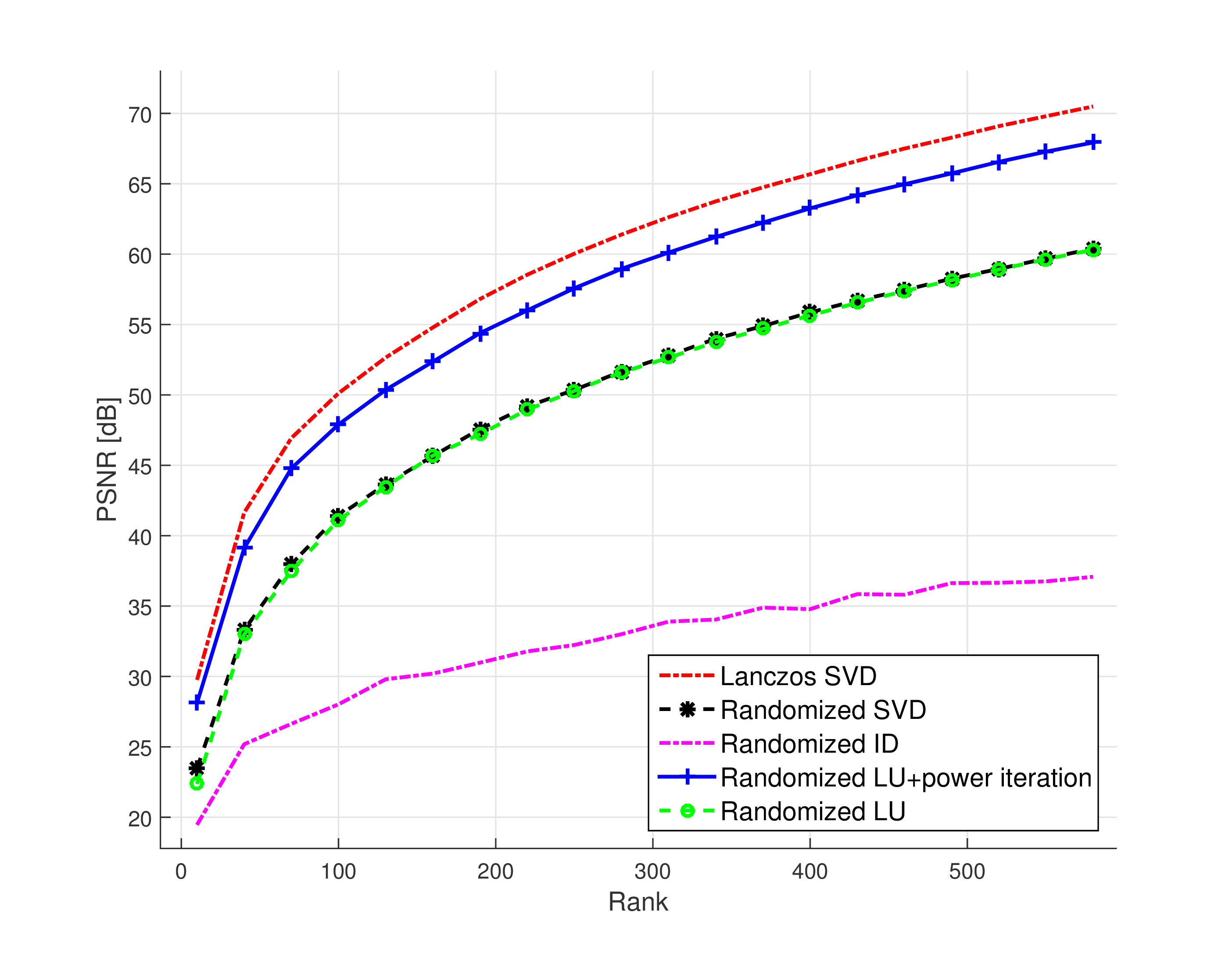}
	\caption{PSNR values from image reconstruction application 
using randomized LU, randomized ID,  randomized SVD and Lanczos SVD algorithms.}
	\label{fig:puppies_error}
\end{figure}
\begin{figure}[H]
	\centering
		\includegraphics[scale=0.1]{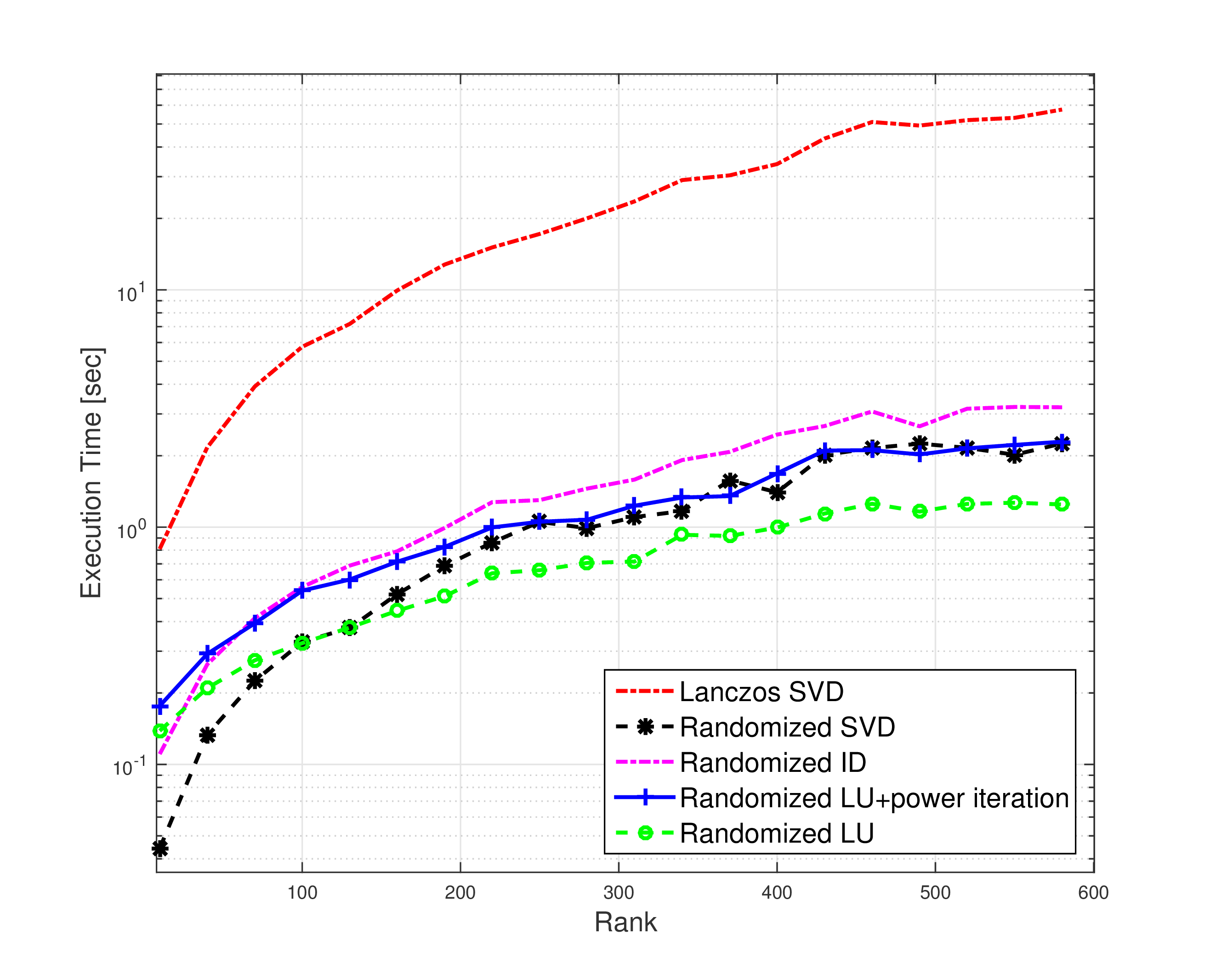}
	\caption{The execution time of the randomized LU, randomized ID,  randomized SVD and Lanczos SVD algorithms.}
	\label{fig:puppies_time}
\end{figure}

\subsection{Fast Randomized LU}
In order to compare the decomposition
running time for Algorithms \ref{alg:randomized_lu} and \ref{alg:fast_randomized_lu}, we apply these algorithms to different
matrix sizes.

The y-axis in Fig. \ref{fig:randLUvsfastRandLU} is the time (in
seconds) for decomposing an $n \times n$ matrix with $l =
3\log_2^2n$ where $n$ is the x-axis.
\begin{figure}[H]
	\centering
	\includegraphics[scale=0.12]{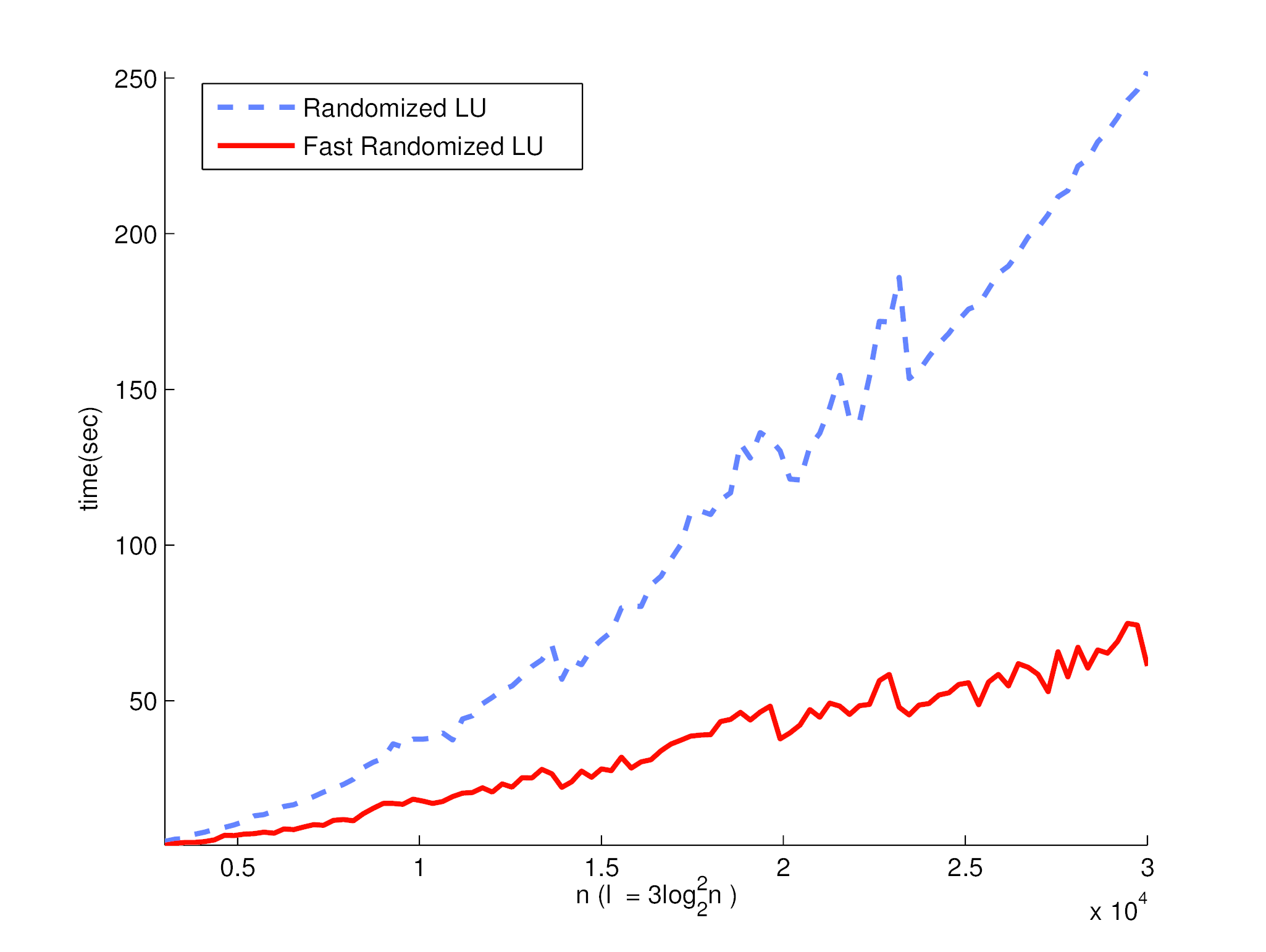}
	\caption{Running time of the fast randomized LU and the randomized LU algorithms}
	\label{fig:randLUvsfastRandLU}
\end{figure}
In addition, we see in Fig. \ref{fig:randLUvsfastRandLUErr} that the error from Algorithm  \ref{alg:fast_randomized_lu} is larger than the error that Algorithm \ref{alg:randomized_lu} generates. Both errors decrease at the same rate. Figure \ref{fig:randLUvsfastRandLUErr}, like Fig. \ref{fig:random_error_single}, shows the relative error (Eq. \eqref{eq:relErr}) for a randomly chosen matrix of size $3000 \times 3000$ with exponentially decaying singular values where $l = k+3$ for different $k$ values.
\begin{figure}[H]
 	\centering
 	\includegraphics[scale=0.1]{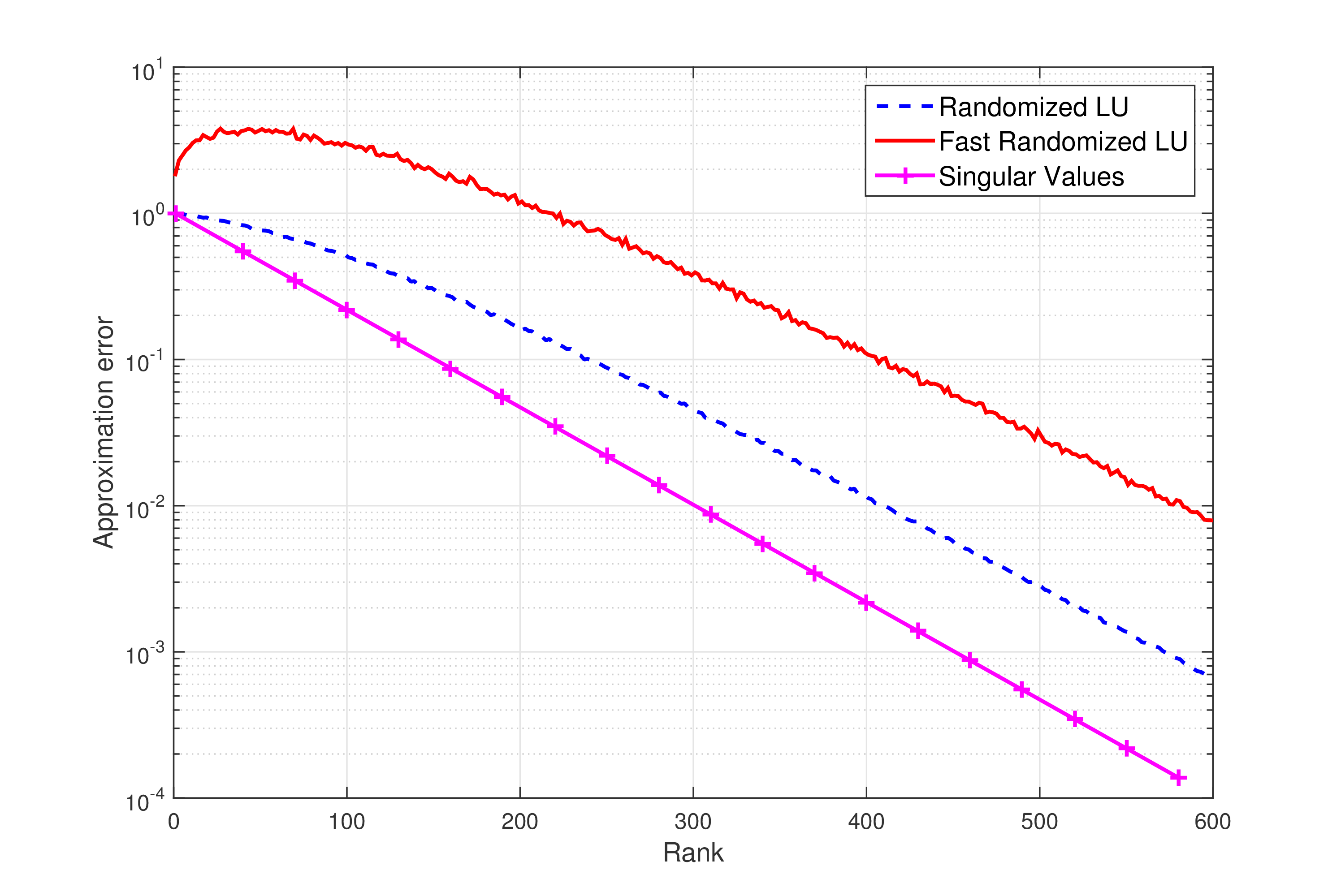}
 	\caption{The normalized error (Eq. \eqref{eq:relErr}) from the fast randomized LU and the randomized LU algorithms.}
 	\label{fig:randLUvsfastRandLUErr}
\end{figure}

The experiment from section \ref{sub:error1} was repeated, with a slowly decaying singular values. The decay of the singular values is the same as was used for the power iterations comparison $\sigma_k=\frac{100}{(9+k)^2}$. These results appear in Fig. \ref{fig:randLUvsfastRandLUErrSlow}
\begin{figure}[H]
	\centering
	\includegraphics[scale=0.1]{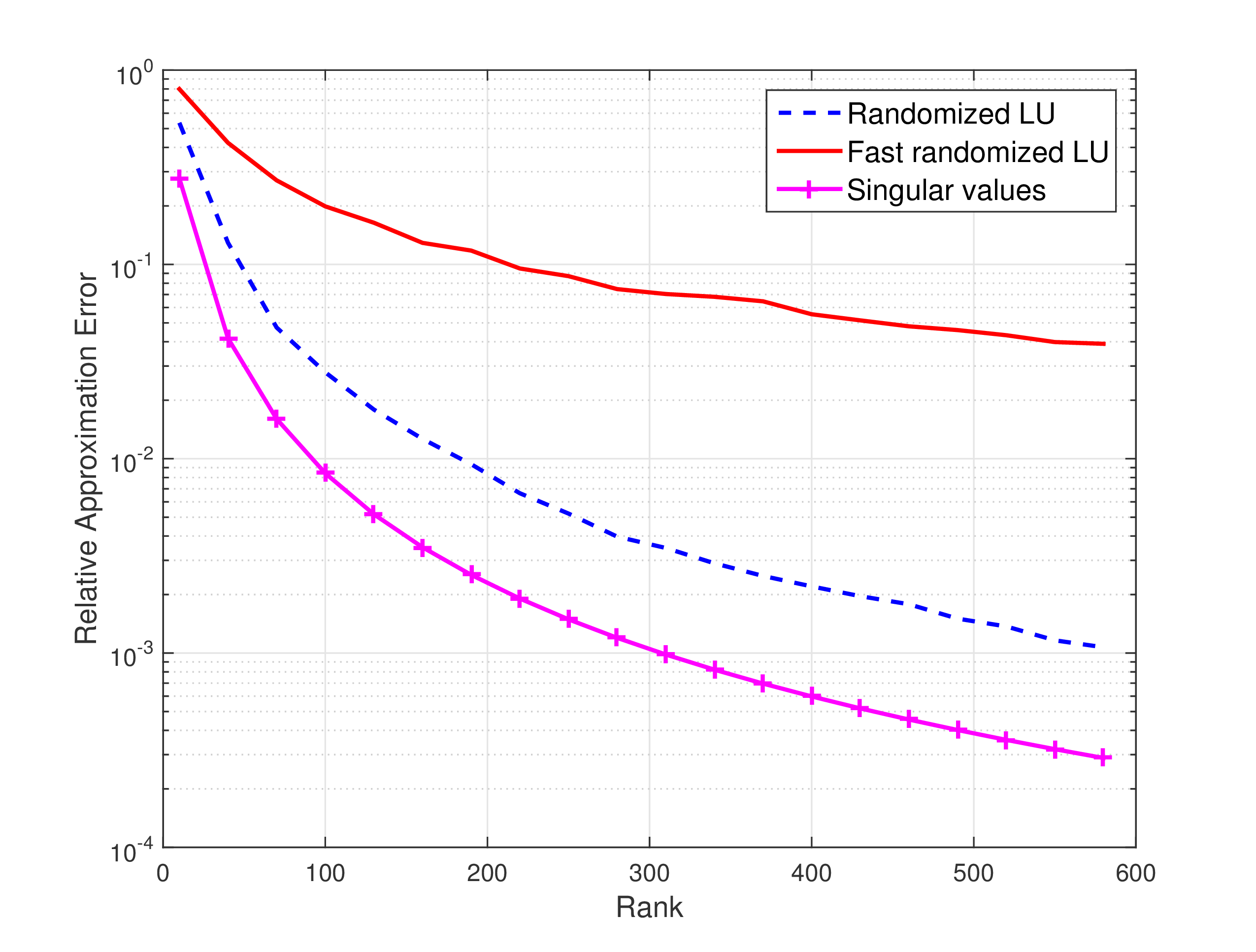}
	\caption{The normalized error from the fast randomized LU and the randomized LU algorithms for slowly decaying singular values.}
	\label{fig:randLUvsfastRandLUErrSlow}
\end{figure}
The reason the error of the fast randomized LU is larger than the error of the randomized LU is due to the fact that the space from which the projections are chosen is much smaller than the space created by the Gaussian-based random projection. The space is smaller since the projection matrix contains only a random diagonal matrix and a random column selection matrix. This large error is also reflected in the error bounds of the algorithm (Theorem \ref{trm:fastRand_lu_err}) and also in the need for a larger $l$ compared to $k$ (Lemma \ref{lem:SRFTerrbnd}).

\section*{Conclusion}
\label{sec:conclusion}
In this work, we presented a randomized algorithm for computing an LU rank $k$ decomposition. 
Given an integer $k$, the algorithm finds an LU decomposition where both $L$ and $U$ are of rank $k$ with negligible failure probability. 
Error bounds for the approximation of the input matrix were derived, and were proved to be proportional to the ($k+1$)th singular value.
The performance of the algorithm (error and computational time) was compared to the randomized SVD, randomized ID and to the application of Lanczos SVD. 
We also showed that the algorithm can be parallelized since it consists mostly of matrix multiplication and pivoted LU. 
The results on GPU show that it is possible to reduce the computational time  
significantly by even using only the standard MATLAB libraries.

\section*{Acknowledgment}
\noindent This research was partially supported by the Israel Science
Foundation (Grant No. 1041/10), by the
Israeli Ministry of Science \& Technology (Grants No. 3-9096, 3-10898), by
US - Israel Binational Science Foundation (BSF 2012282) 
and by a Fellowship from Jyv\"{a}skyl\"{a} University.
The authors would like to thank Yoel Shkolnisky for the helpful discussions.

\bibliographystyle{plain}
\bibliography{FastRandomizedLU}
\end{document}